\RequirePackage{fix-cm}
\documentclass[11pt, a4paper, oneside, pagebackref, final]{amsart}
\usepackage[notref,notcite]{showkeys}
\usepackage[latin1]{inputenc}
\usepackage[T1]{fontenc}
\usepackage{amssymb}
\usepackage[stretch=10]{microtype}
\usepackage{mathtools}
\usepackage[DIV=11, headinclude=true]{typearea}
\usepackage[hidelinks, linktoc=all]{hyperref}

\providecommand{\href}[2]{#2}

\providecommand*{\backref}{}
\providecommand*{\backrefalt}{}
\renewcommand*{\backref}[1]{}
\renewcommand*{\backrefalt}[4]{%
	\ifcase #1 %
	\or
	  Cited page~#2.
	\else
	  Cited pages~#2.
	\fi
}

\newcommand\MTkillspecial[1]{
  \bgroup
  \catcode`\&=9
  \let\\\relax%
  \scantokens{#1}%
  \egroup
}
\newcommand\DeclarePairedDelimiterMultiline[3]{
  \DeclarePairedDelimiter{#1}{#2}{#3}
  \reDeclarePairedDelimiterInnerWrapper{#1}{star}{
    \mathopen{##1\vphantom{\MTkillspecial{##2}}\kern-\nulldelimiterspace\right.}
    ##2
    \mathclose{\left.\kern-\nulldelimiterspace\vphantom{\MTkillspecial{##2}}##3}}
}

\DeclareMathOperator{\Pbb}{\mathbb{P}}
\newcommand{\E}{\mathbb{E}}

\newcommand{\boE}{\mathcal{E}}
\newcommand{\boS}{\mathcal{S}}
\newcommand{\boC}{\mathcal{C}}

\newcommand{\given}{\mid}
\newcommand{\Z}{\mathbb{Z}}
\newcommand{\F}{\mathbb{F}}
\newcommand{\N}{\mathbb{N}}

\DeclarePairedDelimiterMultiline{\abs}{\lvert}{\rvert}
\DeclarePairedDelimiterMultiline{\pare}{(}{)}
\DeclarePairedDelimiterMultiline{\norm}{\lVert}{\rVert}
\DeclarePairedDelimiterMultiline{\card}{\lvert}{\rvert}

\newcommand{\st}{\::\:}

\DeclareMathOperator{\Inf}{Inf}

\newcommand{\coloneqq}{\mathrel{\mathop:}=}

\renewcommand{\epsilon}{\varepsilon}

\renewcommand{\phi}{\varphi}
\renewcommand{\leq}{\leqslant}
\renewcommand{\geq}{\geqslant}

\newcommand\fm{^{\vphantom{-}}}

\newtheorem{thm}{Theorem}[section]
\newtheorem{prop}[thm]{Proposition}
\newtheorem{definition}[thm]{Definition}
\newtheorem{lem}[thm]{Lemma}
\newtheorem{cor}[thm]{Corollary}
\newtheorem{claim}[thm]{Claim}
\newtheorem*{prop*}{Proposition}

\theoremstyle{definition}

\newtheorem{rmk}[thm]{Remark}

\numberwithin{equation}{section}

\title[Random walks without moment condition]{Exponential bounds for random walks on hyperbolic spaces
without moment conditions}
\author{S\'ebastien Gou\"ezel}

\address{IRMAR, CNRS UMR 6625,
Universit\'e de Rennes 1, 35042 Rennes, France}
\email{sebastien.gouezel@univ-rennes1.fr}

\date{\today}

\begin{document}

\begin{abstract}
We consider nonelementary random walks on general hyperbolic spaces.
Without any moment condition on the walk, we show that it escapes linearly
to infinity, with exponential error bounds. We even get such exponential
bounds up to the rate of escape of the walk. Our proof relies on an
inductive decomposition of the walk, recording times at which it could go
to infinity in several independent directions, and using these times to
control further backtracking.
\end{abstract}

\maketitle

\section{Introduction}

Let $X$ be a Gromov-hyperbolic space, with a fixed basepoint $o$. Fix a
discrete probability measure $\mu$ on the space of isometries of $X$. We
assume that $\mu$ is \emph{non-elementary}: in the semigroup generated by the
support of $\mu$, there are two loxodromic elements  with disjoint fixed
points. Let $g_0,g_1,\dotsc$ be independent isometries of $X$ distributed
according to $\mu$. One can then define a random walk on $X$ given by $Z_n
\cdot o$, where $Z_n = g_0\dotsm g_{n-1}$.

In general, results in the literature fall into two classes, qualitative and
quantitative, where the second class requires more stringent assumptions on
the walk.

Without any moment assumption, it is known that $Z_n\cdot o$ converges almost
surely to a point on the boundary $\partial X$, thanks to a beautiful
non-constructive argument originally due to
Furstenberg~\cite{furstenberg_convergence} in a matrix setting but that works
in our setting when $X$ is proper, and extended to the general situation
above by Maher and Tiozzo~\cite{maher_tiozzo}. The idea is to use a
stationary measure on the boundary of $X$ and the martingale convergence
theorem there to obtain the convergence of the random walk. When $X$ is not
proper, the boundary is not compact, and showing the existence of a
stationary measure on the boundary is a difficult part
of~\cite{maher_tiozzo}. In this article, the authors also show linear
progress, in the following sense: there exists $\kappa>0$ such that, almost
surely, $\liminf d(o, Z_n \cdot o)/n \geq \kappa$.

Assuming additional moments conditions, one gets stronger results.
\cite{maher_tiozzo} shows that, if $\mu$ has finite support, then $\Pbb(d(o,
Z_n \cdot o) \leq \kappa n)$ is exponentially small, for some $\kappa > 0$
(we say that the walk makes linear progress with exponential decay). The
finite support assumption has been weakened to an exponential moment
condition in~\cite{sunderland}. More recently, still under an exponential
moment condition,~\cite{boulanger_mathieu_sert_sisto} shows (among many other
results) that the exponential bound holds for any $\kappa$ strictly smaller
than the escape rate $\ell = \lim \E(d(o, Z_n \cdot o))/n$.

When $X$ is a hyperbolic group, one has in fact linear progress with
exponential decay without any moment assumption: this follows from
nonamenability of the group, and the fact that the cardinality of balls is at
most exponential. This arguments breaks down when the space is non-proper,
though, as in many interesting examples such as the curve complex.

Our goal in this paper is to show that, to have linear progress with
exponential decay (even in its strongest versions), there is no need for any
moment condition. Define the escape rate of the walk $\ell(\mu) = \lim
\E(d(o, Z_n \cdot o))/n$ if $\mu$ has a moment of order $1$, i.e., $\sum
\mu(g) d(o, g\cdot o) < \infty$, and $\ell(\mu) = \infty$ otherwise.

Our first result is that the escape rate is positive, with an exponential
error term.
\begin{thm}
\label{thm:0} Consider a discrete non-elementary measure on the space of
isometries of a Gromov-hyperbolic space $X$ with a basepoint $o$. Then there
exists $\kappa>0$ such that, for all $n$,
\begin{equation*}
  \Pbb(d(o, Z_n \cdot o) \leq \kappa n) \leq e^{-\kappa n}.
\end{equation*}
\end{thm}

One recovers in particular that $\ell(\mu)>0$, a fact already proved
in~\cite{maher_tiozzo}. The control in the previous theorem can in fact be
established up to the escape rate:
\begin{thm}
\label{thm:A} Under the assumptions of Theorem~\ref{thm:0}, consider $r <
\ell(\mu)$. Then there exists $\kappa > 0$ such that, for all $n$,
\begin{equation*}
  \Pbb(d(o, Z_n \cdot o) \leq r n) \leq e^{-\kappa n}.
\end{equation*}
\end{thm}
In particular, when $\mu$ has no moment of order $1$, this implies that $d(o,
Z_n \cdot o)/n \to +\infty$ almost surely.

We also get the corresponding statement concerning directional convergence to
infinity. For $\xi \in \partial X$ and $x,y\in X$, denote the corresponding
Gromov product by
\begin{equation}
\label{eq:def_grom_prod_infty}
  (x,\xi)_y = \inf_{z_n\to \xi} \liminf_{n} (x, z_n)_y,
\end{equation}
where $(x, z_n)_y = (d(y, x) + d(y, z_n) - d(x, z_n))/2$ is the usual Gromov
product inside the space (see Section~\ref{sec:prerequisites} for more
background on Gromov-hyperbolic spaces). The limit only depends on the choice
of the sequence $z_n$ up to $2\delta$. Intuitively, $(x,\xi)_y$ is the
distance from $y$ to a geodesic between $x$ and $\xi$. It is also the amount
that $x$ has moved in the direction of $\xi$ compared to $y$. A sequence
$x_n$ converges to $\xi$ if and only if $(x_n, \xi)_o \to \infty$.

\begin{thm}
\label{thm:B} Under the assumptions of Theorem~\ref{thm:A}, $Z_n \cdot o$
converges almost surely to a point $Z_\infty \in \partial X$. Moreover, for
any $r<\ell(\mu)$, there exists $\kappa > 0$ such that, for all $n$,
\begin{equation*}
  \Pbb((Z_n \cdot o, Z_\infty)_o \leq r n) \leq e^{-\kappa n}.
\end{equation*}
\end{thm}
Theorem~\ref{thm:B} readily implies Theorem~\ref{thm:A} as $(Z_n \cdot o,
Z_\infty)_o \leq d(o, Z_n\cdot o)$, which follows directly from the
definition.

The convergence statement in Theorem~\ref{thm:B} is due
to~\cite{maher_tiozzo}. The novelty is the quantitative exponential bound,
without any moment assumption. Note that, in both theorems, when $\mu$ has no
moment of order $1$, one may take any $r\geq 0$, so the conclusion is
superlinear growth with exponential decay.

It follows from subadditivity that, for any $r \leq \ell$, the sequence
$-\log(\Pbb(d(o, Z_n \cdot o) \leq r n))/n$ converges to a limit $I(r)$. This
is a rate function in the classical sense of large deviations in probability
theory. Theorem~\ref{thm:A} shows that the rate function is strictly positive
for $r<\ell$, recovering part
of~\cite[Theorem~1.2]{boulanger_mathieu_sert_sisto} while removing their
exponential moment assumption. Note that~\cite{boulanger_mathieu_sert_sisto}
also obtains exponential estimates for upper deviation inequalities
$\Pbb(d(o, Z_n\cdot o) \geq r n)$ for $r > \ell$. These estimates can not
hold without exponential moments, since exponential controls for lower and
upper deviation probabilities imply an exponential moment for the measure,
see~\cite[Subsection~3.1]{boulanger_mathieu_sert_sisto}.

\begin{rmk}
The fact that we use discrete measures in the above theorems is for
convenience only, to avoid discussing measurability issues and conditioning
on zero measure sets. Suitable versions removing discreteness, but adding
measurability and separability conditions, hold with the same proofs.
\end{rmk}

\medskip

Our approach is elementary, in the spirit of~\cite{mathieu_sisto}
and~\cite{boulanger_mathieu_sert_sisto} (the latter article is a strong
inspiration for our work), and does not rely on any boundary theory. The main
intuition is the following. In the hyperbolic plane, we define a path as
follows: walk straight on during a distance $d_1$, then turn by an angle
$\theta_1 \leq \bar \theta<\pi$, then walk straight on during a distance
$d_2$, then turn by an angle $\theta_2 \leq \bar \theta$, and so on. If all
the lengths $d_i$ are larger than a constant $D = D(\bar \theta)$, then this
path is essentially going straight to infinity, and at time $n$ it is roughly
at distance $d_1 + \dotsb + d_n$ of the origin. The problem when doing a
random walk is that the analogues of the angles $\theta_i$ could be equal to
$\pi$, i.e., the walker could come back exactly along its footsteps. But this
should not happen often. Our main input is a technical way to justify that
indeed it does not happen often, in a precise quantitative version: we will
keep track of some times (called \emph{pivotal times} below) at which the
random walk can choose some direction, with most choices leading to progress
towards infinity (this is implemented through the notion of Schottky set
coming from~\cite{boulanger_mathieu_sert_sisto}), and at which we will keep
some degree of freedom in an inductive construction. Of course, backtracking
can happen later on, and we will spend the degree of freedom we had kept to
still control the behavior after backtracking.

We could give directly the proof of Theorem~\ref{thm:B}, but it would be very
hard to follow. Instead, we will start with proofs of easier statements, and
add new ingredients in increasingly complicated proofs.
Section~\ref{sec:free_group} is devoted to the simplest instance of our
proof, in the free group, where everything is as transparent as possible.
Then, Section~\ref{sec:prerequisites} introduces some tools of
Gromov-hyperbolic geometry (notably chains, shadows and Schottky sets) that
will be used to extend the previous proof to a non-tree setting.
Section~\ref{sec:linear_escape} uses these tools in a crude way to prove
Theorem~\ref{thm:0}, i.e., linear escape with exponential decay, and also
convergence at infinity with exponential bounds. Section~\ref{sec:precise}
follows the same strategy but in a more refined way, to get
Theorems~\ref{thm:A} and~\ref{thm:B}.

\section{Linear escape with exponential decay on free groups}
\label{sec:free_group}

The goal of this section is to illustrate the concept of pivotal times in the
simplest possible setting. We show that, for a class of measures without
moments on the free group, there is linear escape with exponential decay. Of
course, this follows from non-amenability. Instead of the result, what
matters here is the proof: the rest of the paper is an extension of the same
idea to technically more involved contexts (general measures,
Gromov-hyperbolic spaces), but the main insight can be explained much more
transparently in a tree setting.

\begin{thm}
\label{thm:free_group} Let $d \geq 3$. Let $\mu$ be a probability measure on
$\F_d$ that can be written as $\mu_S * \nu$, where $\mu_S$ is the uniform
probability measure on the canonical generators of $\F_d$, and $\nu$ is a
probability measure with $\nu(e) = 0$. Let $Z_n = g_1 \dotsm g_n$, where the
$g_i$ are independent and distributed according to $\mu$. There exists
$\kappa > 0$ (independent of $\nu$ and of $d$) such that, for all $n$,
\begin{equation*}
  \Pbb(\abs{Z_n} \leq \kappa n) \leq e^{-\kappa n}.
\end{equation*}
\end{thm}

\begin{rmk}
The fact that $\kappa$ can be chosen independently of $\nu$ and of $d$ does
not follow from non-amenability, and is really a byproduct of our proof
technique.
\end{rmk}

\begin{rmk}
The restrictions $d\geq 3$ and $\nu(e) = 0$ are simplifying assumptions to
have a proof that is as streamlined as possible. In the next sections, we
will prove analogous theorems but for general measures, on general hyperbolic
spaces.
\end{rmk}

The key point in the proof of Theorem~\ref{thm:free_group} is the next lemma.

\begin{lem}
\label{lem:borne_abs} There exists $\kappa > 0$ satisfying the following.
Consider $d\geq 3$ and $n\geq 0$. Fix $w_1,\dotsc, w_n$ nontrivial words in
$\F_d$, and let $Z_n = s_1 w_1 \dotsm s_n w_n$, where the $s_i$ are
generators of $\F_d$, chosen uniformly and independently. Then
$\Pbb(\abs{Z_n} \leq \kappa n) \leq e^{-\kappa n}$.
\end{lem}
This lemma directly implies Theorem~\ref{thm:free_group}, by conditioning
with respect to the realizations of $\nu$ and just keeping the randomness
coming from the factor $\mu_S$ in $\mu = \mu_S * \nu$.

To prove the lemma, one wants to argue that the walk does not backtrack too
much. Of course, the walk can backtrack completely: as the size of the $w_i$
is not controlled, it may happen that $w_n$ is exactly inverse to $s_1
w_1\dotsm s_n$ and therefore that $Z_n = e$. However, this is unlikely to
happen for most choices of $s_1,\dotsc, s_n$.

A difficulty is that the distance to the origin is not well-behaved under the
walk. For instance, assume that $Z_{n-2} = e$, that $w_{n-1}$ is very long
(of length $2n$, say) and that for some generators $s$ and $t$, one has $tw_n
= (sw_{n-1})^{-1}$. Then $Z_{n-1}$ is far away from the origin, and in
particular it satisfies the inequality $\abs{Z_{n-1}} > n$. However, $Z_n$ is
equal to the origin if $s_{n-1} = s$ and $s_n = t$, which happens with
probability $1/(2d)^2$. This is not exponentially small, even though the
distance control at time $n-1$ is good.

For this reason, we will not try to control inductively the distribution of
the distance to the origin. Instead, we will control a number of branching
points of the random walk up to time $n$, that we call \emph{pivotal points}.
In the general case of random walks in hyperbolic spaces, the definition will
be quite involved, but for trees one can give a direct definition as follows.
Denote by $\gamma_n$ the path in the Cayley graph of $\F_d$ corresponding to
the walk up to $Z_n$, i.e., the concatenation of the geodesics from $e$ to
$s_1$ then to $s_1 w_1$ then to $s_1 w_1 s_2$ and so on until $s_1 w_1 s_2
w_2\dotsm s_n w_n = Z_n$.

\begin{definition}
A time $k \in [1,n]$ is a pivotal time (with respect to $n$) if $s_k$ is the
inverse neither of the last letter of $Z_{k-1}$, nor of the first letter
$(w_k)_0$ of $w_k$ (so that the path $\gamma_n$ is locally geodesic of length
$3$ around $Z_{k-1}$) and moreover the path $\gamma_n$ does not come back to
$Z_{k-1} s_k$ afterwards.

We will denote by $P_n$ the set of pivotal times with respect to $n$.
\end{definition}

In other words, $k$ is pivotal if the walk at time $k$ goes away from the
origin during two steps ($s_k$ and then $(w_k)_0$) and then remains stuck in
the subtree based at $Z_{k-1} s_k (w_k)_0$.

The evolution of the set of pivotal times is not monotone: if the walk
backtracks a lot, then many times that were pivotal with respect to $n$ will
not be any more pivotal with respect to $n+1$, since the non-backtracking
condition is not satisfied any more. On the other hand, the only possible new
pivotal point is the last one: $P_{n+1} \subseteq P_n \cup \{n+1\}$.

We will say that a sequence $(s'_1, \dotsc, s'_n)$ is \emph{pivoted} from
$\bar s = (s_1,\dotsc, s_n)$ if they have the same pivotal times and,
additionally, $s'_k = s_k$ for all $k$ which is not a pivotal time. This is
an equivalence relation. Moreover, a sequence has many pivoted sequences: if
$k$ is a pivotal time and one changes $s_k$ to $s'_k$ which still satisfies
the local geodesic condition (i.e., $s'_k$ is different from the last letter
of $Z_{k-1}$ and from the first letter of $w_k$), then we claim that
$(s_1,\dotsc, s'_k, \dotsc, s_n)$ is pivoted from $(s_1,\dotsc, s_n)$.
Indeed, the part of $\gamma_n$ originating from $Z_{k-1} s_k (w_k)_{0}$ never
comes back on the edge from $Z_{k-1}$ to $Z_{k-1} s_k$ (not even on its
endpoints), so changing $s_k$ to $s'_k$ does not change this fact. Thus the
behavior of $\gamma'_n$ after $Z_{k-1}$ is exactly the same as that of
$\gamma_n$, but in a different subtree -- one has pivoted the end of
$\gamma_n$ around $Z_{k-1} s_k$, hence the name. In particular, subsequent
pivotal times are the same. Moreover, since the trajectory never comes back
before $Z_{k-1}s_k$, pivotal times before $k$ are not affected, and are the
same for $\gamma_n$ and $\gamma'_n$.

More generally, denoting the pivotal times by $p_1 < \dotsb < p_q$, then
changing the $s_{p_i}\fm$ to $s'_{p_i}$ still satisfying the local geodesic
condition gives a pivoted sequence. Let $\boE_n(\bar s)$ be the set of
sequences which are pivoted from $\bar s$. Conditionally on $\boE_n(\bar s)$,
the previous discussion shows that the random variables $s'_{p_i}$ are
independent (but not identically distributed as each of them is drawn from
some subset of the generators depending on $i$, of cardinality $\card{S}-1$
or $\card{S}-2$).

\begin{prop}
\label{prop:pivotA} Let $A_n = \card{P_n}$ be the number of pivotal times.
Then, in distribution, $A_{n+1} \geq A_n + U$ where $U$ is a random variable
independent from $A_n$ and distributed as follows:
\begin{align*}
  &\Pbb(U = -j) = \frac{2d-3}{d (2d-2)^j} \text{ for $j > 0$},\\
  &\Pbb(U = 0)  = 0,\\
  &\Pbb(U = 1)  = \frac{d-1}{d}.
\end{align*}
In other words, $\Pbb(A_{n+1} \geq i) \geq \Pbb(A_n + U \geq i)$ for all $i$.
\end{prop}
\begin{proof}
Let us fix a sequence $\bar s = (s_1,\dotsc, s_n)$, and let $q = \card{P_n}$
be its number of pivotal times.  We will prove the estimate by conditioning
on $\boE_n(\bar s)$. Let $\bar s' \in \boE_n(\bar s)$.

First, assume there are no pivotal points, i.e., $q=0$. Then for each $\bar
s'$ there are at least $2d-2$ generators which are different from the last
letter of $Z'_n$ and from the first letter of $w_{n+1}$, giving rise to one
pivotal time in $P'_{n+1}$ with probability at least $(2d-2)/(2d) =
\Pbb(U=1)$. Otherwise, $\card{P'_{n+1}}= 0$. Conditionally on $\boE_n(\bar
s)$, it follows that the conclusion of the lemma holds.

Assume now that there is at least one pivotal point. From the last pivotal
time onward, the behavior is the same over all the equivalence class
$\boE_n(\bar s)$, so the last letter of $Z'_n$ does not depend on $\bar s'$.
There are at least $2d-2$ generators of $\F_d$ which are different from the
last letter of $Z'_n$ and from the first letter of $w_{n+1}$. If $s'_{n+1}$
is such a generator, then $P'_{n+1} = P'_n \cup \{n+1\}$. Therefore,
\begin{equation*}
  \Pbb(A_{n+1} \geq q + 1 \given \boE_n(\bar s)) \geq (2d-2)/(2d).
\end{equation*}
We have adjusted the definition of $U$ so that the right hand side is $\Pbb(U
\geq 1)$.

Fix now $s'_{n+1}$ which is not such a nice generator. Then $s'_{n+1}
w_{n+1}\fm$ may backtrack, possibly until the last pivotal point $Z'_{p_q}$,
thereby decreasing the number of pivotal points with respect to $n+1$.
However, it may only backtrack further if the generator $s'_{p_q}$ is exactly
the inverse of the corresponding letter in $w_{n+1}$. This can happen for
$s'$, but then it will not happen for all the pivoted configurations of $s'$
obtained by changing $s'_{p_q}$ to another generator still satisfying the
local geodesic condition. Therefore,
\begin{equation*}
  \Pbb(A_{n+1} \leq q-2 \given \boE_n(\bar s)) \leq \frac{2}{2d} \times \frac{1}{2d-2},
\end{equation*}
where the first factor corresponds to the choice of a generator $s'_{n+1}$
which does not satisfy the local geodesic condition, and the second factor
corresponds to the choice of the specific generator for $s'_{p_q}$ to make
sure that one backtracks further.

More generally, to cross $j$ pivotal times, there is one specific choice of
generator at each of these pivotal times, which can only happen with a
probability at most $1/(2d-2)$ at each of these times. Therefore, for $j\geq
1$,
\begin{equation*}
  \Pbb(A_{n+1} \leq q -j \given \boE_n(\bar s)) \leq \frac{2}{2d} \cdot \frac{1}{(2d-2)^{j-1}}.
\end{equation*}
We have adjusted the distribution of $U$ so that the right hand side is
exactly $\Pbb(U \leq -j)$.

Finally, we obtain the inequalities
\begin{align*}
  &\Pbb(A_{n+1} \leq q - j \given \boE_n(\bar s)) \leq \Pbb (U \leq -j) \text{ for $j > 0$},\\
  &\Pbb(A_{n+1} \geq q + 1 \given \boE_n(\bar s)) \geq \Pbb (U \geq 1).
\end{align*}
Taking the complement in the first inequality yields $\Pbb(A_{n+1} \geq q + k
\given \boE_n(\bar s)) \geq \Pbb(U \geq k)$ for all $k \in \Z$. As $A_n$ is
constant equal to $q$ on $\boE_n(\bar g)$, the right hand side is $\Pbb(A_n +
U \geq q + k \given \boE_n(\bar s))$. Writing $i = q + k$, we have obtained
for all $i$ the inequality
\begin{equation*}
  \Pbb(A_{n+1} \geq i \given \boE_n(\bar s)) \geq \Pbb(A_n + U \geq i \given \boE_n(\bar s)).
\end{equation*}
As this inequality is uniform over the conditioning, it gives the conclusion
of the lemma.
\end{proof}

\begin{proof}[Proof of Lemma~\ref{lem:borne_abs}]
Let $U_1, U_2, \dotsc$ be a sequence of i.i.d.\ random variables distributed
like $U$ in Proposition~\ref{prop:pivotA}. Iterating the proposition, one
gets $\Pbb(A_n \geq k) \geq \Pbb(U_1+\dotsb+U_n \geq k)$. The random
variables $U_i$ have an exponential moment. Moreover, their expectation is
positive when $d\geq 3$, as it is $(2d - 5)\cdot(d - 1)/((2d - 3)\cdot d)$.
Large deviations for sums of i.i.d.\ real random variables with an
exponential moment ensure the existence of $\kappa
> 0$ such that $\Pbb(U_1+\dotsb+U_n \leq \kappa n) \leq e^{-\kappa n}$ for
all $n$. Then $\Pbb(A_n \leq \kappa n) \leq e^{-\kappa n}$. As the distance
to the origin is bounded from below by the number of pivotal points, this
proves Lemma~\ref{lem:borne_abs}, except that the constant $c$ depends on the
number of generators $d$. However, the random variables $U = U(d)$ depending
on $d$ increase with $d$ (in the sense that when $d\geq d'$ then $\Pbb(U(d)
\geq k) \geq \Pbb(U(d') \geq k)$ for all $k$). Therefore, one can use the
random variables $U(3)$ to obtain a lower bound in all free groups $\F_d$
with $d\geq 3$.
\end{proof}

The rest of the paper is devoted to the extension of this argument to general
measures and general Gromov-hyperbolic spaces. While the intuition will
remain the same, the definition of pivotal times will need to be adjusted, as
there is no well-defined concept of subtree: instead, we will use a suitable
notion of shadow, and require that the walk after the pivotal time remains in
the shadow. Also, to separate possible directions, we will rely on the notion
of Schottky sets introduced by~\cite{boulanger_mathieu_sert_sisto}, instead
of just using the generators as in the free group. These notions are
explained in the next section.

\section{Prerequisites on Gromov-hyperbolic spaces}

\label{sec:prerequisites}

Let $X$ be a metric space, and $x,y,z\in X$. Their Gromov product is defined
by
\begin{equation*}
  (x, z)_y = \frac{1}{2}\pare*{d(x, y) + d (y, z) - d(x, z)}.
\end{equation*}
Let $\delta \geq 0$. A metric space is $\delta$-Gromov hyperbolic if, for all
$x,y,z,a$,
\begin{equation}
\label{eq:hyperb_ineq}
  (x,z)_a \geq \min((x,y)_a, (y,z)_a) - \delta.
\end{equation}
When the space is geodesic, this is equivalent (up to changing $\delta$) to
the fact that geodesic triangles are thin, i.e., each side is contained in
the $\delta$-neighborhood of the other two sides.

In the rest of the paper, $X$ is a $\delta$-hyperbolic metric space (without
any geodesicity or properness or separability condition). We also fix a
basepoint $o \in X$.

\subsection{Boundary at infinity}

We recall a few basic facts on the boundary at infinity of a
Gromov-hyperbolic space that we will need later on.

A sequence $(x_n)_{n\in \N}$ is converging at infinity if $(x_n, x_m)_o$
tends to infinity when $m, n\to \infty$. Two sequences $(x_n)$ and $(y_n)$
which are converging at infinity are converging to the same limit if $(x_n,
y_n)_o \to \infty$. This is an equivalence relation, thanks to the
hyperbolicity inequality. Quotienting by this equivalence relation, one gets
the boundary at infinity of the space $X$ denoted $\partial X$.

The $C$-shadow of a point $x$, seen from $o$, is the set of points $y$ such
that $(y, o)_x \leq  C$. We denote it with $\boS_o(y; C)$. Geometrically,
this means that a geodesic from $o$ to $y$ goes within distance $C +
O(\delta)$ of $x$. Let us record a few classical properties of shadows.

\begin{lem}
\label{lem:dist_in_shadow} For $y\in \boS_o(x; C)$, one has $d(y, o) \geq
d(x, o) - 2C$.
\end{lem}
\begin{proof}
We have
\begin{equation*}
  d(y, o) = d(y, x) + d(x, o) -2 (y, o)_x \geq 0 + d(x, o) -2 C. \qedhere
\end{equation*}
\end{proof}

\begin{lem}
\label{lem:convergence_of_mem_shadow} Let $C>0$, and let $x_n \in X$ be such
that $d(o, x_n) \to \infty$. Consider another sequence $y_p$ such that, for
all $n$, eventually $y_p \in \boS_o(x_n; C)$. Then $y_p$ converges at
infinity.
\end{lem}
\begin{proof}
Fix $n$ large. For large enough $p$, one has $y_p \in \boS_o(x_n; C)$, i.e.,
$(o, y_p)_{x_n} \leq C$. As $(o, y_p)_{x_n} + (x_n, y_p)_o = d(o, x_n)$, this
gives $(x_n, y_p)_o \geq d(o, x_n)-C$.

For large enough $p, q$, we get (using hyperbolicity for the first
inequality)
\begin{equation}
\label{eq:qsdfqsdfds}
  (y_p, y_q)_o \geq \min((y_p, x_n)_o, (y_q, x_n)_o) - \delta
  \geq d(o, x_n) - C - \delta.
\end{equation}
As $d(o, x_n)\to \infty$ by assumption, it follows that $(y_p, y_q)_o \to
\infty$, as claimed.
\end{proof}

\begin{lem}
\label{lem:grom_prod_of_mem_shadow} Let $C>0$ and $x\in X$. Consider $y \in
\boS_o(x; C)$, and a point $\xi \in
\partial X$ which is a limit of points in $\boS_o(x; C)$. Then
\begin{equation*}
  (y, \xi)_o \geq d(o, x) - C -3\delta.
\end{equation*}
\end{lem}
\begin{proof}
Let $z_n \in \boS_o(x; C)$ be a sequence converging to $\xi$. As the Gromov
product at infinity does not depend on the sequence up to $2\delta$, we have
$(y, \xi)_o \geq \liminf (y, z_n)_o - 2\delta$. Moreover, as both $y$ and
$z_n$ belong to $\boS_o(x; C)$, the inequality~\eqref{eq:qsdfqsdfds} gives
$(y, z_n)_o \geq d(o,x)-C-\delta$. The conclusion follows.
\end{proof}

\subsection{Chains and shadows}
\label{subsec:chain}

In a hyperbolic space, $(x, z)_y$ is roughly the distance from $y$ to a
geodesic between $x$ and $z$. In particular, if $(x, z)_y \leq C$ for some
constant $C$, this means that the points $x, y, z$ are roughly aligned in
this order, up to an error $C$. We will say that the points are
\emph{$C$-aligned}.

In a hyperbolic space, if in a sequence of points all consecutive points are
$C$-aligned, and the points are separated enough, then the sequence is
progressing linearly, and all points in the sequence are
$C+O(\delta)$-aligned (see for
instance~\cite[Theorem~5.3.16]{ghys_hyperbolique}). We will need variations
around this classical idea.

We start with distance estimates for 3 points.

\begin{lem}
\label{lem:3_points}. Consider $x,y,z$ with $(x, z)_y\leq C$. Then $d(x,z)
\geq d(x, y)-C$ and $d(x,z)\geq d(y, z)-C$.
\end{lem}
\begin{proof}
By symmetry, it suffices to prove the first inequality. We claim that $d(x,z)
\geq d(x, y) - (x, z)_y$, which implies the result. Expanding the definition
of the Gromov product, this inequality holds if and only if
\begin{equation*}
  \frac{d(y, x) + d(y, z) - d(x, z)}{2} + d(x,z) \geq d(x, y).
\end{equation*}
This reduces to $d(y,z) + d(x,z) \geq d(x,y)$, which is just the triangular
inequality.
\end{proof}

The next lemma gives estimates for 4 points, from which results for more
points will follow by induction.

\begin{lem}
\label{lem:4_points} Consider $w,x,y,z\in X$, and $C\geq 0$. Assume $(w, y)_x
\leq C$ and $(x, z)_y \leq C+\delta$ and $d(x, y)\geq 2C+2\delta+1$. Then
$(w, z)_x \leq C+\delta$.
\end{lem}
\begin{proof}
By definition of the Gromov product, $(x, z)_y + (y, z)_x = d(x, y)$. As $(x,
z)_y \leq C + \delta$, we get $(y, z)_x \geq d(x,y) - C - \delta$. As $d(x,
y)\geq 2C+2\delta+1$, this gives $(y, z)_x \geq C+\delta+1$. Writing down the
first condition and the hyperbolicity condition, we get
\begin{equation*}
  C \geq (w, y)_x \geq \min ((w, z)_x, (z, y)_x) - \delta.
\end{equation*}
If the minimum were realized by $(z, y)_x$, we would get $C \geq (C+\delta+1)
- \delta$, a contradiction. Therefore, the minimum is realized by $(w, z)_x$,
which gives $(w, z)_x \leq C+\delta$.
\end{proof}

\begin{definition}
For $C,D\geq 0$, a sequence of points $x_0,\dotsc, x_n$ is a $(C,D)$-chain if
one has $(x_{i-1}, x_{i+1})_{x_i} \leq C$ for all $0 < i < n$, and $d(x_i,
x_{i+1}) \geq D$ for all $0 \leq i < n$.
\end{definition}

\begin{lem}
\label{lem:chaine} Let $x_0,\dotsc, x_n$ be a $(C, D)$ chain with $D \geq 2C
+ 2\delta+1$. Then $(x_0,x_n)_{x_1} \leq C + \delta$, and
\begin{equation}
\label{eq:dx0xn}
  d(x_0,x_n) \geq \sum_{i=0}^{n-1} (d(x_i, x_{i+1}) - (2C+2\delta)) \geq n.
\end{equation}
\end{lem}
\begin{proof}
Let us show by decreasing induction on $i$ that $(x_{i-1}, x_n)_{x_i} \leq
C+\delta$, the result being true for $i=n-1$ by assumption. Assume it holds
for $i+1$. Then the points $x_{i-1}, x_i, x_{i+1}, x_n$ satisfy the
assumptions of Lemma~\ref{lem:4_points}, which gives $(x_{i-1}, x_n)_{x_i}
\leq C+\delta$ as desired.

Let us now show that $d(x_j, x_n) \geq \sum_{i=j}^{n-1} (d(x_i, x_{i+1}) -
(2C+2\delta))$ by decreasing induction on $j$, the case $j=n$ being trivial
and the case $j=0$ being~\eqref{eq:dx0xn}. We have
\begin{equation*}
  d(x_j, x_n) = d(x_j, x_{j+1}) + d(x_{j+1}, x_n) - 2 (x_j, x_n)_{x_{j+1}}
  \geq d(x_j, x_{j+1}) + d(x_{j+1}, x_n) - (2C+2\delta),
\end{equation*}
which concludes the induction.
\end{proof}

\begin{lem}
\label{lem:chaine'} Let $x_0,\dotsc, x_n$ be a $(C, D)$ chain with $D \geq 2C
+ 4\delta+1$. Then for all $i$ one has $(x_0, x_n)_{x_i} \leq C + 2\delta$.
\end{lem}
\begin{proof}
Lemma~\ref{lem:chaine} applied to the $(C, D)$-chain $x_i, x_{i+1},\dotsc,
x_n$ gives $(x_i, x_n)_{x_{i+1}} \leq C+\delta$. The same lemma applied to
the $(C, D)$-chain $x_{i+1}, x_i,\dotsc, x_0$ gives $(x_{i+1}, x_0)_{x_i}
\leq C+\delta$. Therefore, the points $x_0, x_i, x_{i+1}, x_n$ are
$(C+\delta)$-aligned. Let us apply Lemma~\ref{lem:4_points} to these points,
with $C+\delta$ instead of $C$. It gives $(x_0, x_n)_{x_i} \leq C + 2\delta$,
as claimed.
\end{proof}

We will need to say that a point $z$ belongs to a half-space based at a point
$y$ and directed towards a point $y^+$. The usual definition for this is the
shadow of $y^+$ seen from $y$, defined as the set $\boS_y(y^+; C)$ of points
$z$ with $(y, z)_{y^+}\leq C$ for some suitable $C$. Unfortunately, this
definition is not robust enough for our purposes as we will need to say that
being in a half-space and walking again from $y$ one stays in the half-space,
which is not satisfied by this definition due to the loss of $\delta$ when
one applies the hyperbolicity inequality.

A more robust definition can be given in terms of chains. If we have a chain
(which goes roughly in a straight direction by the previous lemma) and if we
prescribe the direction of its first jump, then we are essentially
prescribing the direction of the whole chain. This makes it possible to
define another notion that we call chain-shadow, as follows. The choice of
the minimal distance $2C+2\delta+1$ between points in the chain in this
definition is somewhat arbitrary, it should just be large enough that lemmas
on the linear progress of chains apply.

\begin{definition}
Let $C\geq 0$ and $y, y^+, z\in X$. We say that $z$ belongs to the
$C$-chain-shadow of $y^+$ seen from $y$ if there exists a $(C,
2C+2\delta+1)$-chain $x_0 = y,x_1,\dotsc, x_n = z$ satisfying additionally
$(x_0, x_1)_{y^+} \leq C$. We denote the chain-shadow with $\boC\boS_y(y^+;
C)$.
\end{definition}

The next lemma shows that this definition of shadow is roughly equivalent to
the usual definition in terms of the Gromov product $(y, z)_{y^+}$.

\begin{lem}
\label{lem:grom_prod_of_half_space} If $z \in \boC\boS_y(y^+; C)$, then $(y,
z)_{y_+} \leq 2C+\delta$ and $d(y,z) \geq d(y,y^+) -2C-\delta$.
\end{lem}
\begin{proof}
Let $x_0=y,x_1,\dotsc, x_n=z$ be a $(C, 2C+2\delta+1)$-chain as in the
definition of chain-shadows. We have
\begin{equation*}
  d(y, z) = d(y, x_1) + d(x_1, z) - 2 (y, z)_{x_1}
  = d(y, y^+) + d(y^+, x_1) - 2 (y, x_1)_{y_+} + d(x_1, z) - 2 (y, z)_{x_1}.
\end{equation*}
Let us bound $(y, x_1)_{y_+}$ with $C$ (by the definition of chain-shadows)
and $(y, z)_{x_1}$ by $C+\delta$ (thanks to Lemma~\ref{lem:chaine} applied to
the chain $x_0,\dotsc, x_n$). Let us also bound from below $d(y^+, x_1) + d
(x_1, z)$ with $d(y^+, z)$. We get
\begin{equation*}
  d(y, z) \geq d(y, y^+) + d(y^+, z) - 4C-2\delta.
\end{equation*}
Expanding the definition of the Gromov product, this gives $(y, z)_{y^+} \leq
2C+\delta$. Then we get $d(y,z) \geq d(y, y^+) - 2C-\delta$ by applying
Lemma~\ref{lem:3_points} to $y, y^+, z$.
\end{proof}

\subsection{Schottky sets}

To be able to prescribe enough directions at pivotal points, we will use a
variation around the notion of Schottky set
in~\cite{boulanger_mathieu_sert_sisto}. This is essentially a finite set of
isometries such that, for all $x$ and $y$, most of these isometries put $x$
and $sy$ in general position with respect to $o$, i.e., such that $x, o, sy$
are $C$-aligned for some given $C$.

\begin{definition}
Let $\eta, C, D\geq 0$. A finite set $S$ of isometries of $X$ is $(\eta, C,
D)$-Schottky if
\begin{itemize}
\item For all $x, y \in X$, we have $\card{\{s \in S, (x, s y)_o \leq C\}}
    \geq (1-\eta) \card S$.
\item For all $x, y \in X$, we have $\card{\{s \in S, (x, s^{-1} y)_o \leq
    C\}} \geq (1-\eta) \card S$.
\item For all $s \in S$, we have $d(o, so)\geq D$.
\end{itemize}
\end{definition}

We could define analogously a notion of an $(\eta, C, D)$-probability
measure, where the previous definition would be this property for the uniform
measure on $S$.

The next proposition shows that one can find Schottky sets by using powers of
two loxodromic isometries.

\begin{prop}
\label{prop:exists_schottky} Fix two loxodromic isometries $u$ and $v$ of
$X$, with disjoint sets of fixed points at infinity. For all $\eta>0$, there
exists $C>0$ such that, for all $D>0$, there exist $n\in \N$ and an $(\eta,
C, D)$-Schottky set in $\{w_1\dotsm w_n \st w_i \in \{u,v\}\}$.
\end{prop}
\begin{proof}
This is essentially a classical application of the ping-pong method.
\cite[Proposition A.2]{boulanger_mathieu_sert_sisto} contains a slightly less
precise statement, but their proof also gives our stronger version, as we
explain now. Let $S_n = \{w_1\dotsm w_n \st w_i \in \{u,v\}\}$.

The ping-pong argument at infinity shows that one can choose $n$ large enough
so that, for all $m$ the elements $w_1 \dotsm w_m$ for $w_i \in \{u^n, v^n\}$
are all different, loxodromic, with disjoint sets of fixed points at
infinity. Let us fix such an $n$, and then such an $m$ with $2^{-m}<\eta/2$,
and denote these $2^m$ isometries with $g_1,\dotsc, g_{2^m}$. They all belong
to $S_{nm}$. Let $g_i^+$ and $g_i^-$ be their attractive and repulsive fixed
points.

Let $K$ be large enough. Define a neighborhood $V(g_i^+) = \{x \in X \st (x,
g_i^+)_o \geq K\}$ and a smaller neighborhood $V'(g_i^+) = \{x \in X \st (x,
g_i^+)_o \geq K + \delta\}$. In the same way, define $V(g_i^-)$ and
$V'(g_i^-)$. If $K$ is large enough, then the $2^{m+1}$ sets
$(V(g_i^\pm))_{i=1,\dotsc, 2^m}$ are disjoint as the fixed points at infinity
of the $g_i$ are all different. Moreover, for large enough $p$, then $g_i^p$
maps the complement of $V(g_i^-)$ to $V'(g_i^+)$, and the complement of
$V(g_i^+)$ to $V'(g_i^-)$.

We claim that, for all $D$, if $p$ is large enough, then $S = \{g_1^p,
\dotsc, g_{2^m}^p\}$ is an $(\eta, K+\delta, D)$-Schottky set. As all these
elements belong to $S_{nmp}$, this will prove the theorem. First, the
condition $d(o, so)\geq D$ for $s = g_i^p$ is true if $p$ is large enough, as
$g_i$ is loxodromic. Let us show that $\card{\{s \in S, (x, s y)_o \leq
K+\delta\}} \geq (1-\eta) \card S$ for all $x,y$ (the corresponding
inequality with $s^{-1}$ is similar). There is at most one $s = g_i$ for
which $y \in V(g_i^-)$, as all these sets are disjoint. There is also at most
one $s=g_j$ for which $x \in V(g_j^+)$, again by disjointness. If $s=g_k$ is
not one of these two, we claim that $(x, sy)_o \leq K+\delta$. This will
prove the result, since this implies
\begin{equation*}
  \card{\{s \in S, (x, s y)_o \leq K+\delta\}}
  \geq \card S - 2
  = 2^m - 2
  = \card S (1- 2\cdot 2^{-m})
  \geq (1-\eta)\card S.
\end{equation*}
As $x \notin V(g_k^+)$, we have $(x, g_k^+)_o < K$. As $y \notin V(g_k^-)$,
we have $s y = g_k y \in V'(g_k^+)$, i.e., $(s y, g_k^+)_o \geq K+\delta$. By
hyperbolicity, we obtain
\begin{equation*}
  K > (x, g_k^+)_o \geq \min((x, sy)_o, (sy, g_k^+)_o) - \delta.
\end{equation*}
(Note that the hyperbolicity inequality~\eqref{eq:hyperb_ineq}, initially
stated inside the space, remains true for the Gromov product at infinity as
we have used an $\inf$ in its definition~\eqref{eq:def_grom_prod_infty}). If
the minimum were realized by $(sy, g_k^+)_o \geq K+\delta$, we would get $K >
(K+\delta)-\delta$, a contradiction. Therefore, the minimum is realized by
$(x, sy)_o$, yielding $K > (x, sy)_o - \delta$ as claimed.
\end{proof}

\begin{cor}
\label{cor:exists_Schottky} Let $\mu$ be a non-elementary discrete measure on
the set of isometries of $X$. For all $\eta>0$, there exists $C>0$ such that,
for all $D>0$, there exist $M>0$ and an $(\eta, C, D)$-Schottky set in the
support of $\mu^M$.
\end{cor}
\begin{proof}
By definition of a non-elementary measure, one can find loxodromic elements
$u_0$ and $v_0$ with disjoint fixed points in the support of $\mu^a$ and
$\mu^b$ for some $a, b > 0$. Then $u=u_0^b$ and $v = v_0^a$ belong to the
support of $\mu^{ab}$ and have disjoint fixed points. Applying
Proposition~\ref{prop:exists_schottky}, we obtain an $(\eta, C, D)$-Schottky
set in the support of $\mu^{abn}$ as desired.
\end{proof}

\section{Linear escape}

\label{sec:linear_escape}

In this section, we prove Theorem~\ref{thm:0}, i.e., the random walk on $X$
driven by a non-elementary measure escapes linearly towards infinity, with
exponential bounds. We copy the proof of Section~\ref{sec:free_group},
replacing subtrees with chain-shadows in the definition of pivotal times, and
generators with elements of a Schottky set. The reader who would prefer to
use shadows instead of chain-shadows may do so for intuition, but should be
warned that the argument will then barely fail (at a single place, the
backtracking step in the proof of Lemma~\ref{lem:rembobine}).

Like in Section~\ref{sec:free_group}, the main technical part is to
understand what happens for walks of the form $w_0 s_1 w_1\dotsm w_{n-1}s_n
w_n$, where the $w_i$ are fixed, while the $s_i$ are random, and drawn from a
Schottky set. This will be done in Subsection~\ref{subsec:pivotal}, while the
application to prove Theorem~\ref{thm:0} is done in
Subsection~\ref{subsec:thm0}

\subsection{A simple model}
\label{subsec:pivotal}

In this section, we fix isometries $w_0, w_1,\dotsm$ of $X$, a constant
$C_0>0$, and $S$ a $(1/100, C_0, D)$-Schottky set of isometries of $X$. We
will assume that $D$ is large enough compared to $C_0$ (for definiteness $D
\geq 20C_0+100\delta+1$ will do). Let $\mu_S$ be the uniform measure on $S$.
Let $s_i$ be i.i.d.~random variables distributed like $\mu_S^2$.

We form a random process on $X$ by composing the $w_i$ and $s_i$ and applying
them to the basepoint $o$. Our goal is to understand the behavior of
$y_{n+1}^- = w_0 s_1 w_1 \dotsm s_n w_n \cdot o$ when $n$ tends to infinity.
The main result of this subsection is the following proposition.

\begin{prop}
\label{prop:linear_escape_model} There exists a universal constant $\kappa >
0$ (independent of everything) such that, for all $n$,
\begin{equation*}
  \Pbb(d(o, y_{n+1}^-) \leq \kappa n) \leq e^{-\kappa n}.
\end{equation*}
\end{prop}

Write  $s_i = a_i b_i$ with $a_i, b_i \in S$. We define
\begin{align*}
  & y_i^- = w_0 s_1 w_1 \dotsm s_{i-1} w_{i-1} \cdot o,
  \quad y_i = w_0 s_1 w_1\dotsm w_{i-1} a_i\cdot o,
  \quad y_i^+ = w_0 s_1 w_1 \dotsm w_{i-1} a_i b_i\cdot o,
\end{align*}
the three points visited during the transition around $i$.  We have $d(y_i^-,
y_i\fm) = d(o, a_i \cdot o)\geq D$ as $a_i$ belongs to the $(1/100, C_0,
D)$-Schottky set $S$. In the same way, $d(y_i\fm, y_i^+) \geq D$. A
difficulty that we will need to handle is that $d(y_i^+, y_{i+1}^-)$ may be
short, as there is no lower bound on $w_i$, while we need long jumps
everywhere to apply the results on chains of Subsection~\ref{subsec:chain}.

We will define a sequence of pivotal times $P_n \subseteq \{1,\dotsc, n\}$,
evolving with time: when going from $n$ to $n+1$, we will either add a
pivotal time at time $n+1$ (so that $P_{n+1} = P_n \cup \{n+1\}$, if the walk
is going more towards infinity), or we will remove a few pivotal times at the
end because the walk has backtracked (in this case, $P_{n+1} = P_n \cap
\{1,\dotsc, m\}$ for some $m$).

Let us define inductively the pivotal times, starting from $P_0 = \emptyset$.
Assume that $P_{n-1}$ is defined, and let us define $P_n$. Let $k = k(n)$ be
the last pivotal time before $n$, i.e., $k = \max(P_{n-1})$. (If $P_{n-1} =
\emptyset$, take $k=0$ and let $y_k = o$ -- we will essentially ignore the
minor adjustments to be made in this special case in the forthcoming
discussion). Let us say that the local geodesic condition is satisfied at
time $n$ if
\begin{equation}
\label{eq:local_geodesic}
  (y_k, y_n)_{y_n^-} \leq C_0, \quad (y_n^-, y_n^+)_{y_n} \leq C_0, \quad (y_n\fm, y_{n+1}^-)_{y_n^+} \leq C_0.
\end{equation}
In other words, the points $y_k\fm, y_n^-, y_n\fm, y_n^+, y_{n+1}^-$ follow
each other successively, with a $C_0$-alignment condition. As the points are
well separated by the definition of Schottky sets, this will guarantee that
we have a chain, progressing in a definite direction.

If the local geodesic condition is satisfied at time $n$, then we say that
$n$ is a pivotal time, and we set $P_n = P_{n-1} \cup \{n\}$. Otherwise, we
backtrack to the largest pivotal time $m \in P_{n-1}$ for which $y_{n+1}^-$
belongs to the $(C_0+\delta)$-chain-shadow of $y_m^+$ seen from $y_m$. In
this case, we erase all later pivotal times, i.e., we set $P_n = P_{n-1} \cap
\{1,\dotsc, m\}$. If there is no such pivotal time $m$, we set $P_n =
\emptyset$.

\begin{lem}
\label{lem:dist} Assume that $P_n$ is nonempty. Let $m$ be its maximum. Then
$y_{n+1}^-$ belongs to the $(C_0+\delta)$-chain-shadow of $y_m^+$ seen from
$y_m$.
\end{lem}
\begin{proof}
If $P_n$ has been defined from $P_{n-1}$ by backtracking, then the conclusion
of the lemma is a direct consequence of the definition. Otherwise, the last
pivotal time is $n$. In this case, let us show that $y_{n+1}^-$ belongs to
the $(C_0 + \delta)$-chain-shadow of $y_n^+$ seen from $y_n$, by considering
the chain $y_n, y_{n+1}^-$. By definition of the chain-shadow, we should
check that $(y_n\fm, y_{n+1}^-)_{y_n^+} \leq C_0+\delta$ and $d(y_n\fm,
y_{n+1}^-) \geq 2C_0+4\delta +1$. The first inequality is obvious as
$(y_n\fm, y_{n+1}^-)_{y_n^+} \leq C_0\leq C_0+\delta$ by the local geodesic
condition~\eqref{eq:local_geodesic}. Moreover, since $(y_n\fm,
y_{n+1}^-)_{y_n^+}\leq C_0$ by~\eqref{eq:local_geodesic},
Lemma~\ref{lem:3_points} gives $d(y_n\fm, y_{n+1}^-) \geq d(y_n\fm, y_n^+) -
C_0 \geq D-C_0$, which is $\geq 2C_0+4\delta +1$ if $D$ is large enough.
\end{proof}

\begin{lem}
\label{lem:Pn_chain} Let $P_n = \{k_1 < \dotsb < k_p\}$. Then the sequence
$y_{k_1}^-, y_{k_1}\fm, y_{k_2}^-, y_{k_2}\fm,\dotsc, y_{k_p}\fm, y_{n+1}^-$
is a $(2C_0+3\delta, D -2C_0 - 3\delta)$-chain.
\end{lem}
\begin{proof}
Let us first check the condition on Gromov products. We have to show that
$(y_{k_{i-1}}, y_{k_i})_{y_{k_i}^-} \leq 2C_0+3\delta$ and $(y_{k_i}^-,
y_{k_{i+1}}^-)_{y_{k_i}} \leq 2C_0+3\delta$. The first inequality is obvious,
as it follows from the first property in the local geodesic condition when
introducing the pivotal time $k_i$. Let us show the second one.
Lemma~\ref{lem:dist} applied to the time $k_{i+1}-1$ shows that
$y_{k_{i+1}}^-$ belongs to the $(C_0+\delta)$ chain-shadow of $y_{k_i}^+$
seen from $y_{k_i}$. Lemma~\ref{lem:grom_prod_of_half_space} thus yields
($y_{k_{i+1}}, y_{k_i})_{y_{k_i}^+} \leq 2C_0+3\delta$. Moreover,
$(y_{k_i}^+, y_{k_i}^-)_{y_{k_i}} \leq C_0$ by the local geodesic condition
when introducing the pivotal time $k_i$. We apply Lemma~\ref{lem:4_points}
with the points $y_{k_i}^-, y_{k_i}\fm, y_{k_i}^+, y_{k_{i+1}}^-$, with $C =
2C_0+2\delta$. As $d(y_{k_i}\fm, y_{k_i}^+) \geq D$ is large enough, this
lemma applies and gives $(y_{k_i}^-, y_{k_{i+1}}^-)_{y_{k_i}} \leq
2C_0+3\delta$. This is the desired inequality.

Let us check the condition on distances. We have to show that $d(y_{k_i}^-,
y_{k_i}\fm) \geq D -2C_0 - 3\delta$ and $d(y_{k_i}\fm, y_{k_{i+1}}^-) \geq D
-2C_0 - 3\delta$. The first condition is obvious as $d(y_{k_i}^-, y_{k_i}\fm)
\geq D$. For the second, Lemma~\ref{lem:grom_prod_of_half_space} gives
$d(y_{k_i}\fm, y_{k_{i+1}}^-) \geq d(y_{k_i}\fm, y_{k_i}^+) - 2C_0 - 3\delta
\geq D - 2C_0 -3\delta$.
\end{proof}

The first point in the previous chain can be replaced with $o$:
\begin{lem}
\label{lem:Pn_chaino} Let $P_n = \{k_1 < \dotsb < k_p\}$. Then the sequence
$o, y_{k_1}\fm, y_{k_2}^-, y_{k_2}\fm,\dotsc, y_{k_p}\fm, y_{n+1}^-$ is a
$(2C_0+4\delta, D - 2C_0 - 3\delta)$-chain.
\end{lem}
\begin{proof}
We have to control $d(o, y_{k_1})$ and $(o, y_{k_2}^-)_{y_{k_1}}$ as the
other quantities are controlled by Lemma~\ref{lem:Pn_chain}. For this, we
will apply Lemma~\ref{lem:4_points} to the points $y_{k_2}^-, y_{k_1}\fm,
y_{k_1}^-, o$ with $C = 2C_0+3\delta$. We have $(y_{k_2}^-,
y_{k_1}^-)_{y_{k_1}} \leq 2C_0+3\delta$ by Lemma~\ref{lem:Pn_chain}, and
$(y_{k_1}, o)_{y_{k_1}^-} \leq C_0$ (this is the first property in the local
geodesic condition when introducing the pivotal time $k_1$), and
$d(y_{k_1}\fm, y_{k_1}^-) \geq D \geq 2C+\delta+1$. Therefore,
Lemma~\ref{lem:4_points} gives $(y_{k_2}^-, o)_{y_{k_1}} \leq 2C_0+4\delta$.
Moreover, Lemma~\ref{lem:3_points} gives
\begin{equation*}
  d(y_{k_1}, o) \geq d(y_{k_1}\fm,
  y_{k_1}^-) - (y_{k_1}, o)_{y_{k_1}^-} \geq D - C_0 \geq D - 2C_0 - 3\delta.
  \qedhere
\end{equation*}
\end{proof}

\begin{prop}
\label{prop:doZn} We have $d(o,y_{n+1}^-) \geq \card{P_n}$.
\end{prop}
\begin{proof}
This follows from Lemma~\ref{lem:Pn_chaino}, saying that we have a chain of
length at least $\card{P_n}$ between $o$ and $y_{n+1}^-$, and from
Lemma~\ref{lem:chaine}, saying that the distance grows linearly along a
chain.
\end{proof}

This proposition shows that, to obtain the linear escape rate with
exponential decay, it suffices to show that there are linearly many pivotal
times.

\begin{lem}
\label{lem:prob_new_free} Fix $s_1,\dotsc, s_n$, and draw $s_{n+1}$ according
to $\mu_S^2$. The probability that $\card{P_{n+1}} = \card{P_n} + 1$ (i.e.,
that $n+1$ gets added as a pivotal time) is at least $9/10$.
\end{lem}
\begin{proof}
In the local geodesic condition~\eqref{eq:local_geodesic}, the last property
reads $(g \cdot o, g b_n w_n \cdot o)_{g b_n \cdot o} \leq C_0$ for $g = w_0
s_1\dotsm w_{n-1} a_n$. Composing with $b_n^{-1}g^{-1}$, it becomes
$(b_n^{-1} \cdot o, w_n \cdot o)_o \leq C_0$. By the definition of a Schottky
set, this inequality is satisfied with probability at least $1-\eta = 99/100$
when choosing $b_n$. Once $b_n$ is fixed, the other two properties in the
geodesic condition only depend on $a_n$, and each of them is satisfied with
probability at least $99/100$, again by the Schottky property. They are
satisfied simultaneously with probability at least $98/100$. As $(99/100)
\cdot (98/100) \geq 9/10$, this concludes the proof.
\end{proof}

The key point is to control the backtracking length. For this, we will see
that for one configuration that backtracks a lot, there are many
configurations that do not. Given $\bar s = (s_1,\dotsc, s_n)$, let us say
that another sequence $\bar s' = (s'_1, \dotsc, s'_n)$ is pivoted from $\bar
s$ if they have the same pivotal times, $b'_k = b_k$ for all $k$, and $a'_k =
a_k$ when $k$ is not a pivotal time.

\begin{lem}
\label{lem:many_pivoted} Let $i$ be a pivotal time of $\bar s = (s_1,\dotsc,
s_n)$. Replace $s_i = a_i b_i$ with $s'_i = a'_i b_i$ which still satisfies
the local geodesic condition~\eqref{eq:local_geodesic} (with $n$ replaced by
$i$). Then $(s_1,\dotsc, s'_i, \dotsc, s_n)$ is pivoted from $\bar s$.
\end{lem}
\begin{proof}
We should show that the pivotal times of $\bar s'$ are the same as those of
$\bar s$. Until time $i$, the sequences are the same, hence they have the
same pivotal times: $P_{i-1}(\bar s) = P_{i-1}(\bar s')$. Then $i$ is added
as a pivotal time for both $\bar s$ and $\bar s'$ by assumption, therefore
$P_i(\bar s) = P_i(\bar s')$. Then the remaining part of the trajectory for
$\bar s$ never backtracks beyond $i$, as $i$ remains a pivotal time. This
backtracking property is defined in terms of the relative position of the
trajectory compared to $y_i$ and $y_i^+$, and therefore it depends on $b_i$
but not on the beginning of the trajectory (and in particular it does not
depend on $a_i$). Hence, replacing $a_i$ with $a'_i$ does not change the
backtrackings, which are the same for $\bar s$ and $\bar s'$ until time $n$.
\end{proof}

Lemma~\ref{lem:many_pivoted} shows that, if a trajectory has $p$ pivotal
times, then it has a lot of pivoted trajectories (exponentially many in $p$)
as one can change $a_i$ to $a'_i$ at each pivotal time. Denote by
$\boE_n(\bar s)$ the set of trajectories which are pivoted from $\bar s$.
Conditionally on $\boE_n(\bar s)$, the random variables $a'_i$ for $i$ a
pivotal time are independent (but not identically distributed, as they are
each drawn from a subset of $S$ depending on $i$, of large cardinality.

\begin{lem}
\label{lem:rembobine} Let $\bar s = (s_1,\dotsc, s_n)$ be a trajectory with
$q$ pivotal times. We condition on $\boE_n(\bar s)$, and we draw $s_{n+1}$
according to $\mu_S^2$. Then, for all $j\geq 0$,
\begin{equation*}
  \Pbb(\card{P_{n+1}} < q-j \given \boE_n(\bar s)) \leq 1/10^{j+1}.
\end{equation*}
\end{lem}
\begin{proof}
If $q=0$, then the result follows readily from Lemma~\ref{lem:prob_new_free}.
Assume $q>0$.

First, the probability that $s_{n+1}$ creates a new pivotal time is at least
$9/10$, by Lemma~\ref{lem:prob_new_free} (and the elements $s_{n+1}$ that
create a new pivotal time are the same over the whole equivalence class
$\boE_n(\bar s)$ as $q>0$). Let us now fix a bad $s_{n+1}$, giving rise to
backtracking.

Let us show the lemma for $j=1$. Let $m<k$ be the last two pivotal times. We
have to show that
\begin{equation}
\label{eq:Pnplusone}
  \Pbb(\card{P_{n+1}} < q-1 \given \boE_n(\bar s), s_{n+1}) \leq 1/10,
\end{equation}
i.e., most trajectories do not backtrack beyond $k$: for many choices of
$a_k$, then $y_{n+1}^-$ should belong to the $(C_0+\delta)$-chain-shadow of
$y_m^+$ seen from $y_m$. By Lemma~\ref{lem:dist} applied at time $k-1$, we
already know that $y_k^-$ belongs to this set. Therefore, there exists a
chain $x_0 = y_m, x_1,\dotsc, x_i = y_k^-$ pointing in the chain-shadow. With
a good choice of $a_k$, we will increase the chain by adding $y_{n+1}^-$ at
its end.

Let us consider $a'_k$ so that the points $x_{i-1}\fm, y_k^-, y_k\fm,
y_{n+1}^-$ are $C_0$-aligned, i.e., such that $(x_{i-1}, y_k)_{y_k^-} \leq
C_0$ and $(y_k^-, y_{n+1}^-)_{y_k} \leq C_0$. By the Schottky property, there
are at least $(98/100)\card{S}$ such $a'_k$. Let us show that, with this
choice, $y_{n+1}^-$ belongs to the chain-shadow of $y_m^+$ seen from $y_m$
(and therefore backtracking stops here). For this, it is enough to see that
$x_0,\dotsc, x_{i-1}, y_k^-, y_{n+1}^-$ is a $(C_0+\delta,
2C_0+4\delta+1)$-chain. We have to see that $d(y_k^-, y_{n+1}^-) \geq
2C_0+4\delta+1$ and $(x_{i-1}\fm, y_{n+1}^-)_{y_k^-} \leq C_0+\delta$. For
this, apply Lemma~\ref{lem:4_points} to the points $x_{i-1}\fm, y_k^-,
y_k\fm, y_{n+1}^-$, which are $C_0$-aligned. As $d(y_k^-, y_k\fm) \geq D$ is
large enough, this lemma gives $(x_{i-1}\fm, y_{n+1}^-)_{y_k^-} \leq
C_0+\delta$. Moreover, Lemma~\ref{lem:3_points} gives $d(y_k^-, y_{n+1}^-)
\geq d(y_k^-, y_k\fm) - (y_k^-, y_{n+1}^-)_{y_k} \geq D-C_0 \geq
2C_0+4\delta+1$, as claimed.

In the equivalence class, the number of possible choices for $a'_k$ when
introducing the pivotal time $k$ is at least $(98/100)\card{S}$, since most
choices satisfy the local geodesic condition (see the proof of
Lemma~\ref{lem:prob_new_free}). The number of choices of $a'_k$ that ensure
there is no further backtracking is also bounded below by $(98/100)\card{S}$,
by the previous discussion, so that the number of bad choices is at most $(1
- (98/100)) \card{S}$. Finally, the proportion of bad choices that lead to
further backtracking is at most
\begin{equation*}
  \frac{(1 - (98/100)) \card{S}}{(98/100)\card{S}}
  < \frac{1}{10}.
\end{equation*}
This proves~\eqref{eq:Pnplusone} for $j=1$.

To prove the lemma for $j=2$, let us fix $s_{n+1}$ as well as a bad choice of
$a'_k$ that gives rise to backtracking beyond $k$ (this happens with
probability at most $1/10$). We have to see that, once these quantities are
fixed, the probability to backtrack past the previous pivotal time is at most
$1/10$. This is the same argument as above. The case of general $j$ is proved
analogously by induction.
\end{proof}

\begin{lem}
\label{lem:AnU2} Let $A_n = \card{P_n}$ be the number of pivotal times. Then,
in distribution, $A_{n+1} \geq A_n + U$ where $U$ is a random variable
independent from $A_n$ and distributed as follows:
\begin{align*}
  &\Pbb(U = -j) = \frac{9}{10^{j+1}} \text{ for $j > 0$},\\
  &\Pbb(U = 0)  = 0,\\
  &\Pbb(U = 1)  = \frac{9}{10}.
\end{align*}
In other words, $\Pbb(A_{n+1} \geq i) \geq \Pbb(A_n + U \geq i)$ for all $i$.
\end{lem}
\begin{proof}
Conditionally on $\boE_n(\bar s)$, this follows from
Lemma~\ref{lem:rembobine}, just like in the proof of
Proposition~\ref{prop:pivotA}: one shows that
\begin{equation*}
  \Pbb(A_{n+1} \geq i \given \boE_n(\bar s)) \geq \Pbb(A_n + U \geq i \given \boE_n(\bar s)).
\end{equation*}
As the inequality is uniform over the conditioning, the unconditioned version
follows.
\end{proof}

\begin{prop}
\label{prop:Pnlinear} There exists a universal constant $\kappa > 0$ such
that, for all $n$,
\begin{equation*}
  \Pbb(\card{P_n} \leq \kappa n) \leq e^{-\kappa n}.
\end{equation*}
\end{prop}
\begin{proof}
Let $U_1,U_2,\dotsc$ be a sequence of independent copies of the variable $U$
from Lemma~\ref{lem:AnU2}. Iterating this lemma gives
\begin{equation*}
  \Pbb(\card{P_n} \geq i) \geq \Pbb(U_1+\dotsb+U_n \geq i)
\end{equation*}
for all $i$. In particular, $\Pbb(\card{P_n} \leq \kappa n) \leq \Pbb(U_1 +
\dotsb + U_n \leq \kappa n)$. As the $U_i$ are real random variables with an
exponential moment and positive expectation, $\Pbb(U_1 + \dotsb + U_n \leq
\kappa n)$ is exponentially small if $\kappa$ is small enough.
\end{proof}

\begin{proof}[Proof of Proposition~\ref{prop:linear_escape_model}]
The linear escape with exponential error term follows from
Proposition~\ref{prop:doZn} giving $d(o, y_{n+1}^-) \geq \card{P_n}$, and
from Proposition~\ref{prop:Pnlinear} ensuring that $\card{P_n}$ grows
linearly outside of a set of exponentially small probability.
\end{proof}

\subsection{Proof of linear escape and convergence at infinity}

\label{subsec:thm0}

Let $\mu$ be a non-elementary measure on the set of isometries of the space
$X$. In this subsection, we prove Theorem~\ref{thm:0}: the $\mu$-random walk
goes to infinity linearly, with an exponential error term. The techniques we
develop along the way will also prove convergence of the walk at infinity.

We apply Corollary~\ref{cor:exists_Schottky} with $\eta=1/100$. Let $C=C_0$
be given by this corollary. Choose $D = D(C_0, \delta)$ large enough so that
the result of the previous Subsection apply ($D=20C_0+100\delta+1$ suffices).
The corollary gives an $(\eta, C_0, D)$ Schottky set $S$ included in the
support of $\mu^M$ for some $M$. For $\alpha>0$ small enough and $N=2M$, we
may write $\mu^N = \alpha \mu_S^2 + (1-\alpha) \nu$ for some probability
measure $\nu$, where $\mu_S$ is the uniform measure on $S$.

As in~\cite[Section~6]{boulanger_mathieu_sert_sisto}, let us reconstruct in a
slightly indirect way the random walk, as follows, on a space $\Omega$
containing Bernoulli random variables $\epsilon_i$ (satisfying
$\Pbb(\epsilon_i=1) = \alpha$ and $\Pbb(\epsilon_i = 0) = 1-\alpha$) and
variables $h_i$ distributed according to $\nu$ and variables $s_i = a_i b_i$
distributed according to $\mu_S^2$, all independent. Define $\gamma_i = s_i$
if $\epsilon_i = 1$, and $\gamma_i = h_i$ if $\epsilon_i = 0$. Then
$\gamma_0\dotsm \gamma_{n-1}$ is distributed like $Z_{Nn}$. With a standard
coupling argument, extending $\Omega$ if necessary, we can also construct on
$\Omega$ a sequence of independent random variables $g_0,g_1,\dotsc$ with
distribution $\mu$ such that $\gamma_i = g_{iN}\dotsm g_{iN + N-1}$.

Let $t_1<t_2<\dotsb$ be the times where $\epsilon_i=1$. Fix $n \in \N$. We
let $\tau=\tau(n)$ be the last index $j$ such that $N(t_j+1) \leq n$, so that
the interval $[Nt_j, N(t_j+1))$ is contained in $[0, n)$. We will decompose
the product $g_0\dotsm g_{n-1}$ as a product of the elements $s'_j = s_{t_j}$
(the product of all $g_i$ for $i \in [Nt_j, N(t_j+1))$) interspersed with
other words that we will consider as fixed, to be in the framework of
Subsection~\ref{subsec:pivotal}. Let $w_j = g_{N(t_{j}+1)} \dotsm g_{Nt_{j+1}
-1}$ (where by convention $t_0 = 0$), and let $w' = w'(n) =
g_{N(t_{\tau(n)}+1)} \dotsm g_{n-1}$ be the last missing word (it really
depends on $n$, contrary to the previous words that just fill the gaps
between blocks corresponding to $\epsilon_j = 1$). By construction,
\begin{equation*}
  Z_n \cdot o = w_0 s'_1 w_1 \dotsm w_{\tau -1} s'_\tau w'(n) \cdot o.
\end{equation*}

We can associate to this decomposition a sequence of pivotal times
$P_1^{(n)}, \dotsc, P_\tau^{(n)}$, where the exponent ${}^{(n)}$ is here to
emphasize that the intermediate words we use depend on $n$. In fact, the only
word that really depends on $n$ is the last word $w' = w'(n)$, as the other
ones are $w_j = g_{(N+1)t_{j}} \dotsm g_{Nt_{j+1} -1}$ so they only depend on
$t_j$. Hence, the sequence of pivotal times is rather
\begin{equation}
\label{eq:pivotal_just_last_step}
  P_1, P_2,\dotsc, P_{\tau-1}, P_\tau^{(n)}.
\end{equation}

The main quantity we will control is
\begin{equation*}
  u_n \coloneqq \card*{P_{\tau(n)}^{(n)}},
\end{equation*}
the final number of pivotal times after $n$ steps of the initial random walk.

\begin{prop}
\label{prop:un_exponential} There exists $\kappa>0$ such that $\Pbb(u_n \leq
\kappa n) \leq e^{-\kappa n}$.
\end{prop}
\begin{proof}
The sequence $t_{j+1}-t_j$ is a sequence of independent random variables with
an exponential tail. Therefore, there exist $C>0$ and $\kappa > 0$ such that
\begin{equation*}
  \Pbb(t_j \geq C j) = \Pbb\pare*{\sum_{i=0}^{j-1} (t_{i+1}-t_i) \geq C j} \leq e^{-\kappa j}.
\end{equation*}
Hence, if $\beta>0$ is small enough, we have $N(t_{\lfloor \beta n
\rfloor}+1) \leq n$ outside of a set with exponentially small probability.
This gives
\begin{equation*}
  \Pbb(\tau(n) \geq \beta n) \leq e^{-\kappa n}
\end{equation*}
for some $\kappa>0$. For any $c>0$, we get
\begin{equation*}
  \Pbb(u_n \leq c n) \leq e^{-\kappa n} + \Pbb(u_n \leq c n, \tau \geq \beta n).
\end{equation*}
Let us concentrate on the second set. We condition with respect to the
$\epsilon_i$ (which fixes the $t_i$, and $\tau$) and with respect to the
$g_i$ outside of the intervals $[Nt_j, N(t_j+1))$ (which fixes the $w_j$ and
$w'$). Once these are fixed, we are in the framework of
Subsection~\ref{subsec:pivotal}. We may therefore apply
Proposition~\ref{prop:Pnlinear} and deduce that, conditionally on these
quantities, we have $\Pbb(u_n \leq c \tau) \leq e^{-c \tau}$, for some $c>0$.
As $\tau \geq \beta n$, this gives conditionally $\Pbb(u_n \leq c \beta n)
\leq e^{-c\beta n}$. As this is uniform on the conditioning, this implies the
conclusion.
\end{proof}

\begin{proof}[Proof of Theorem~\ref{thm:0}]
Outside of a set with exponentially small probability, the number of pivotal
times at the $n$-th step of the random walk is at least $\kappa n$ for some
$\kappa > 0$, by Proposition~\ref{prop:un_exponential}. As the distance to
the origin is bounded below by the number of pivotal times, by
Proposition~\ref{prop:doZn}, this concludes the proof.
\end{proof}

This argument enables us to recover a theorem of~\cite{maher_tiozzo}, the
convergence of the walk at infinity. We even get exponential error terms in
the speed of convergence. We start with a lemma ensuring that positions of
the random walk stay in a shadow.

\begin{lem}
\label{lem:stuck_in_shadow} Let $n \in \N$ and $C>0$. Assume that, for all
$k\geq n$, one has $u_k
> C$. Let $x$ be the position of the walk at the $C$-th pivotal time in
$P_{\tau(n)}^{(n)}$. Then, for all $k\geq n$, the point $Z_k \cdot o$ belongs
to the $(2C_0+6\delta)$-shadow of $x$ seen from $o$.
\end{lem}
\begin{proof}
For $k\geq n$, the set $P_{\tau(k)}^{(k)}$ has strictly more than $C$ points
by assumption. In particular, the $C$-th pivotal time is not introduced at
the last step, and the last step does not backtrack beyond this point. The
set of pivotal times before the last index does not depend on $k$, as
explained before~\eqref{eq:pivotal_just_last_step}. It follows that the
$C$-th pivotal time in $P_{\tau(k)}^{(k)}$ is independent of $k \geq n$. In
particular, $x$ is the position of the walk at a pivotal time in
$P_{\tau(k)}^{(k)}$, for any $k \geq n$.

For $k\geq n$, Lemma~\ref{lem:Pn_chaino} shows that there is a
$(2C_0+4\delta, D - 2C_0-3\delta)$-chain from $o$ to $Z_k \cdot o$ going
through $x$. By Lemma~\ref{lem:chaine'}, we deduce that $(o, Z_k\cdot o)_{x}
\leq 2C_0 + 6 \delta$. In other words, all the points $Z_k \cdot o$ remain in
the $(2C_0+6\delta)$-shadow of $x$ seen from $o$, as claimed.
\end{proof}

\begin{prop}
\label{prop:converge_to_infinity} Almost surely, there is a point $Z_\infty
\in \partial X$ such that $Z_n \cdot o$ converges to $Z_\infty$. Moreover,
there exists $\kappa>0$ such that
\begin{equation}
\label{eq:Pbb_Zinfty}
  \Pbb((Z_n \cdot o, Z_\infty)_o \leq \kappa n) \leq e^{-\kappa n}.
\end{equation}
\end{prop}
\begin{proof}
Fix $c>0$ such that $\Pbb(u_n \leq c n) \leq e^{-c n}$, by
Lemma~\ref{prop:un_exponential}. Since $\Pbb(u_n \leq c n)$ is exponentially
small, Borel-Cantelli ensures that almost surely one has eventually $u_n > c
n$. Lemma~\ref{lem:stuck_in_shadow} then applies, with $C = \lfloor c n
\rfloor - 1$. Let $x_n$ denote the position of the walk at the $(\lfloor c n
\rfloor-1)$-th pivotal time for large $n$. By Proposition~\ref{prop:doZn}, it
satisfies
\begin{equation}
\label{eq:doxn}
  d(o, x_n) \geq \lfloor cn\rfloor - 1.
\end{equation}
The sequence $Z_k \cdot o$ is eventually trapped in the shadow of $x_n$ seen
from $o$ by Lemma~\ref{lem:stuck_in_shadow}. This implies the convergence at
infinity of $Z_k \cdot o$, by Lemma~\ref{lem:convergence_of_mem_shadow}.

\medskip

Finally, let us show the quantitative estimate~\eqref{eq:Pbb_Zinfty}. Assume
that for all $k \geq n$, one has $u_k > ck$ (this happens with probability at
least $1 - C e^{-cn}$). In this case, all the points $Z_k \cdot o$ for $k\geq
n$ belong to the $(2C_0+6\delta)$-shadow of $x_n$. Therefore,
Lemma~\ref{lem:grom_prod_of_mem_shadow} applies and gives
\begin{equation}
\label{eq:linear_lower_bound}
  (Z_n \cdot o, Z_\infty)_o \geq d(o, x_n) - (2C_0+6\delta) - 3\delta.
\end{equation}
Together with~\eqref{eq:doxn}, this gives a linear lower bound for the Gromov
product, that holds outside of an exponentially small set.
\end{proof}

We will also need the following lemma, that follows from the same techniques.

\begin{lem}
\label{lem:uniform_control_away} Let $\mu$ be a non-elementary discrete
measure on the set of isometries of a Gromov-hyperbolic space $X$ with
basepoint $o$. Let $Z_n = g_0\dotsm g_{n-1}$ where the $g_i$ are i.i.d.\ with
distribution $\mu$. Let $\epsilon>0$. There exists $C>0$ such that, for any
isometry $g$,
\begin{equation*}
  \Pbb(\forall n, d(o, gZ_n \cdot o) \geq d (o, g\cdot o) - C) \geq 1-\epsilon.
\end{equation*}
\end{lem}
The point of the lemma is that the possible loss $C$ is uniform in $g$.
Without moment assumptions on $\mu$, it is not possible to get a better
bound, contrary to the case of walks with an exponential moment
(compare~\cite[Theorem~2.12]{boulanger_mathieu_sert_sisto}).
\begin{proof}
We follow the same construction as at the beginning of this subsection to
reconstruct the random walk, but adding the isometry $g$ before the first
step of the random walk. Since the estimates of
Subsection~\ref{subsec:pivotal} are uniform in $w_0$, replacing $w_0$ with $g
w_0$ does not change them. Therefore, the number $u_n \coloneqq
\card*{P_{\tau(n)}^{(n)}}$ of pivotal times for the random walk at time $n$
still satisfies the estimate of Proposition~\ref{prop:un_exponential}: there
exists $\kappa>0$ (independent of $g$) such that $\Pbb(u_n \leq \kappa n)
\leq e^{- \kappa n}$.

Let us fix $n$ such that $\sum_{i\geq n} e^{-\kappa i} < \epsilon/2$. On a
set $A_g$ of probability at least $1-\epsilon/2$ (which may depend on $g$),
one has for all $i \geq n$ the inequality $u_i > \kappa i \geq \kappa n$. As
in the proof of Proposition~\ref{prop:converge_to_infinity}, one can then
find a point $x_n$ such that, for all $i \geq n$, the points $gZ_i \cdot o$
belong to the $(2C_0+6\delta)$-shadow of $x_n$ seen from $o$. In particular,
by Lemma~\ref{lem:dist_in_shadow},
\begin{equation*}
  d(gZ_i\cdot o, o) \geq d(o, x_n) - 4C_0 - 12\delta.
\end{equation*}
Moreover, $x_n$ is of the form $g Z_k \cdot o$ for some $k \leq n$.

By measurability, we can find a set $A$ (independent of $g$) of measure at
least $1-\epsilon/2$ and a constant $C$ such that, for all $\omega \in A$ and
all $k\leq n$, holds $d(o, Z_k \cdot o) \leq C$.

Consider $\omega \in A_g \cap A$ (this set has measure at least
$1-\epsilon$). Then
\begin{equation*}
  d(o, x_n) = d(o, g Z_k \cdot o) \geq d(o, g\cdot o) - d(g\cdot o, g Z_k \cdot o) = d(o, g\cdot o)- d(o, Z_k \cdot o)
  \geq d(o, g \cdot o)-C.
\end{equation*}
For all $i\geq n$, we get $d(gZ_i \cdot o, o) \geq d(o, g \cdot o) - C  -
4C_0 - 12\delta$. For $i < n$, this estimate also holds as $d(o, Z_i \cdot o)
\leq C$. This proves the lemma, for the constant $C + 4C_0+12\delta$ which is
independent of $g$.
\end{proof}

\section{Precise estimates}

\label{sec:precise}

\subsection{A more complicated model}

\label{subsec:complicated}

To obtain precise estimates on the rate of convergence to infinity, we will
need to compare the distance to the origin with the sum of independent real
valued random variables corresponding to the size of jumps of the random
walk. This is done in the next proposition.

\begin{prop}
\label{prop:complicated_model} For $\eta \in (0, 1/100]$, there exists
$\kappa=\kappa(\eta)>0$ with the following property.

Let $S$ be an $(\eta, C_0, D)$-Schottky set of isometries of a
$\delta$-hyperbolic space $X$ with basepoint $o$, where $D$ is large enough
compared to $C_0$ (for definiteness $D \geq 20 C_0+100\delta + 1$ is enough).
Let $\rho_1, \rho_2,\dotsc$ be probability measures on the isometry set of
$X$. Let $R$ be a nonnegative real random variable such that for all $i$ and
all $M \geq 0$ one has
\begin{equation*}
  \Pbb_{\rho_i}(d(o, g\cdot o) \geq M) \geq \Pbb(R \geq M),
\end{equation*}
i.e., the distance with respect to the origin for $\rho_i$ dominates
stochastically $R$, for all $i$.

Let $w_0, w_1, \dotsc$ be fixed isometries of $X$. Let $s_1, s_2,\dotsc$ be
independent random variables, where $s_i$ is sampled according to $\mu_S^2 *
\rho_i * \mu_S^2$. Define $y_{n+1}^- = w_0 s_1 w_1 \dotsm s_n w_n \cdot o$.
Then for all $M\geq 0$,
\begin{equation*}
  \Pbb(d(o, y_{n+1}^-) \leq M) \leq \Pbb(R_1 + \dotsb + R_{\lfloor (1-21\eta)n\rfloor} \leq M) + e^{-\kappa n},
\end{equation*}
where $R_1, R_2, \dotsc$ are independent copies of $R$.
\end{prop}
When all the $\rho_i$ are the Dirac mass at the origin, then the setting of
the proposition is essentially the same as the simple model of
Subsection~\ref{subsec:pivotal}, except that we are sampling the $s_i$
according to $\mu_S^4$ instead of $\mu_S^2$ (which does not really make a
difference). The conclusion in the general setting of
Proposition~\ref{prop:complicated_model} is that the growth rate of the
distance to the origin is at least the growth rate of sums of i.i.d.~random
variables distributed like the $\rho_i$, up to a minor loss (that tends to
$0$ when the proportion $\eta$ of bad elements in the Schottky set tends to
$0$) and an exponentially small error term. This model will be precise enough
to capture the right growth rate of a general random walk, to prove
Theorems~\ref{thm:A} and~\ref{thm:B} in the next paragraphs, in the same way
that we have deduced linear escape with exponential estimates from the
results on the simple model of Subsection~\ref{subsec:pivotal}. The
possibility to have different measures $\rho_i$ at the different jumps will
be important in the application of this proposition in
Subsection~\ref{subsec:precise_moment}, but for the proof the reader may
pretend for simplicity that they are all equal to a fixed measure $\rho$ (and
then one can take $R$ to be the distribution of $d(o, g\cdot o)$ with respect
to $\rho$).

\medskip

To prove Proposition~\ref{prop:complicated_model}, let us introduce a refined
notion of pivotal times, in which we will keep the randomness coming from the
$\rho_i$. Write $s_i = a_i b_i r_i c_i d_i$, where $a_i, b_i, c_i, d_i$ are
distributed according to $\mu_S$ while $r_i$ is distributed according to
$\rho_i$. This gives rise to $6$ successive points at the $i$-th transition:
\begin{align*}
  & y_i^- = y_i^{(0)} = w_0s_1 \dotsm s_{i-1} w_{i-1} \cdot o,
  \quad
  & & y_i^{(1)} = w_0s_1 \dotsm s_{i-1} w_{i-1} a_i\cdot o,
  \\
  &y_i^{(2)} = w_0s_1 \dotsm s_{i-1} w_{i-1} a_ib_i\cdot o,
  & & y_i^{(3)} = w_0s_1 \dotsm s_{i-1} w_{i-1} a_ib_i r_i\cdot o,
  \\
  &y_i = y_i^{(4)} = w_0s_1 \dotsm s_{i-1} w_{i-1} a_i b_i r_i c_i \cdot o,
  & & y_i^+ = y_i^{(5)} = w_0s_1 \dotsm s_{i-1} w_{i-1} a_i b_i r_i c_i d_i \cdot o.
\end{align*}
The distances between two successive points in this list is at least $D$ as
it comes from the application of an element of the Schottky set $S$, except
for the distance between $y_i^{(2)}$ and $y_i^{(3)}$ for which we have no
lower bound as $r_i$ is drawn according to $\rho_i$.

Let us define inductively a set of refined pivotal times, that we will denote
by $\bar P_n$ to differentiate it from the previous unrefined notion. We copy
the definition of Subsection~\ref{subsec:pivotal}. We start from $\bar P_0 =
\emptyset$. Assume that $\bar P_{n-1}$ is defined, and let us define $\bar
P_n$. Let $k = k(n)$ be the last pivotal time before $n$, i.e., $k =
\max(\bar P_{n-1})$. (If $\bar P_{n-1} = \emptyset$, take $k=0$ and let $y_k
= o$). Let us say that the local geodesic condition is satisfied at time $n$
if in the sequence $y_k, y_n^{(0)}, y_n^{(1)}, y_n^{(2)}, y_n^{(3)},
y_n^{(4)}, y_n^{(5)}, y_{n+1}^-$, all successive points are $C_0$-aligned,
and moreover $y_n^{(1)}, y_n^{(3)}, y_n^{(4)}$ are $C_0$-aligned (the latter
condition is useful to compensate the fact that the jump from $y_n^{(2)}$ to
$y_n^{(3)}$ may be small, preventing us to apply the results on chains of
Subsection~\ref{subsec:chain}). If the local geodesic condition is satisfied
at time $n$, then we say that $n$ is a refined pivotal time, and we set $\bar
P_n = \bar P_{n-1} \cup \{n\}$. Otherwise, we backtrack to the largest
refined pivotal time $m \in \bar P_{n-1}$ for which $y_{n+1}^-$ belongs to
the $(C_0+\delta)$ chain-shadow of $y_m^+$ seen from $y_m$. In this case, we
erase all later pivotal times, i.e., we set $\bar P_n = \bar P_{n-1} \cap
\{1,\dotsc, m\}$. If there is no such pivotal time $m$, we set $\bar P_n =
\emptyset$.

For the refined notion, we can prove the analogues of the lemmas of
Subsection~\ref{subsec:pivotal}.

\begin{lem}
\label{lem:dist'} Assume that $\bar P_n$ is nonempty. Let $m$ be its maximum.
Then $y_{n+1}^-$ belongs to the $(C_0+\delta)$ chain-shadow of $y_m^+$ seen
from $y_m$.
\end{lem}
\begin{proof}
The proof is exactly the same as for Lemma~\ref{lem:dist}: when there is
backtracking, this follows from the definition, and when there is no
backtracking (i.e., the last pivotal time is $n$), then the chain $y_n\fm,
y_{n+1}^-$ satisfies all the properties to show that $y_{n+1}^-$ is in the
chain-shadow.
\end{proof}

\begin{lem}
\label{lem:Pn_chain'} Let $\bar P_n = \{k_1 < \dotsb < k_p\}$. Then the
sequence $y_{k_1}^-, y_{k_1}\fm, y_{k_2}^-, y_{k_2}\fm,\dotsc, y_{k_p}\fm,
y_{n+1}^-$ is a $(2C_0+3\delta, D - 2C_0-3\delta)$-chain. Moreover,
$d(y_{k_i}^-, y_{k_i}\fm) \geq d(o, r_{k_i} \cdot o) + D$ for all $i$.
\end{lem}
\begin{proof}
This differs a little bit from the proof of Lemma~\ref{lem:Pn_chain} as there
are more points involved at each pivotal time. It is still basic chain
manipulations, with the only difficulty that the jumps corresponding to $r_i$
and $w_i$ may be short (but since they are surrounded by big jumps with
controlled alignment conditions this can be circumvented easily).

By definition, the points $y_{k_{i-1}}\fm, y_{k_i}^-, y_{k_i}^{(1)},
y_{k_i}^{(2)}, y_{k_i}^{(3)}, y_{k_i}^{(4)}, y_{k_i}^{(5)}$ are
$C_0$-aligned. However, the distances between $y_{k_{i-1}}\fm$ and
$y_{k_i}^-$ on the one hand, and between $y_{k_i}^{(2)}$ and $y_{k_i}^{(3)}$
on the other hand, are not obviously bounded below (contrary to the other
distances, which are $\geq D$), so one can not apply the results on chains to
these points. However, we can fix this by removing one point: we claim that
\begin{equation}
\label{eq:chain_of_removed}
  y_{k_{i-1}}\fm, y_{k_i}^-, y_{k_i}^{(1)}, y_{k_i}^{(3)}, y_{k_i}^{(4)}(=y_{k_i}\fm), y_{k_i}^{(5)}
  \text{ form a $(C_0+\delta, D - 2C_0-3\delta)$ chain.}
\end{equation}

Let us prove this claim. We may apply Lemma~\ref{lem:4_points} to the points
$y_{k_i}^-, y_{k_i}^{(1)}, y_{k_i}^{(2)}, y_{k_i}^{(3)}$, with $C = C_0$, to
deduce that $(y_{k_i}^-, y_{k_i}^{(3)})_{y_{k_i}^{(1)}} \leq C_0+\delta$.
Moreover, Lemma~\ref{lem:3_points} gives $d(y_{k_i}^{(1)}, y_{k_i}^{(3)})
\geq d(y_{k_i}^{(1)}, y_{k_i}^{(2)}) - (y_{k_i}^{(1)},
y_{k_i}^{(3)})_{y_{k_i}^{(2)}} \geq D - C_0$. Moreover, $d(y_{k_{i-1}},
y_{k_i}^-) \geq D - 2C_0-3\delta$ by Lemma~\ref{lem:grom_prod_of_half_space},
as $y_{k_i}^-$ is in the $(C_0+\delta)$ chain shadow of $y_{k_{i-1}}^+$ seen
from $y_{k_{i-1}}$, by Lemma~\ref{lem:dist'}. Finally, note that
$(y_{k_i}^{(1)}, y_{k_i}^{(4)})_{y_{k_i}^{(3)}} \leq C_0$ by the last
assumption in the local geodesic condition. We have checked all the
nontrivial properties in~\eqref{eq:chain_of_removed}, completing its proof.

We have in particular $d(y_{k_{i-1}}\fm, y_{k_i}^-) \geq D - 2C_0-3\delta$,
and also by~\eqref{eq:dx0xn}
\begin{equation}
\label{eq:sdjqsklfdjlmqksdf}
  d(y_{k_i}^-, y_{k_i}\fm)
  = d(y_{k_i}^-, y_{k_i}^{(4)})
  \geq d(y_{k_i}^-, y_{k_i}^{(1)}) + d(y_{k_i}^{(1)}, y_{k_i}^{(3)}) + d(y_{k_i}^{(3)}, y_{k_i}^{(4)})
  - 3 (C_0+\delta).
\end{equation}
By Lemma~\ref{lem:3_points} applied to $y_{k_i}^{(1)}, y_{k_i}^{(2)},
y_{k_i}^{(3)}$,
\begin{equation*}
  d(y_{k_i}^{(1)}, y_{k_i}^{(3)}) \geq d(y_{k_i}^{(2)}, y_{k_i}^{(3)})-(y_{k_i}^{(1)},
y_{k_i}^{(3)})_{y_{k_i}^{(2)}} \geq d(o,r_i \cdot o) - C_0.
\end{equation*}
The two other distances in~\eqref{eq:sdjqsklfdjlmqksdf} are bounded below by
$D$. Using $D \geq 3 (C_0+\delta) + C_0$, we obtain
\begin{equation*}
  d(y_{k_i}^-, y_{k_i}\fm) \geq D + d(o, r_i \cdot o).
\end{equation*}
This proves all the distance conditions in the claim of the lemma.

Let us now check the Gromov product estimates. Applying
Lemma~\ref{lem:chaine} to the chain~\eqref{eq:chain_of_removed}, we get
$(y_{k_{i-1}}, y_{k_i})_{y_{k_i}^-} \leq C_0+2\delta \leq 2C_0+3\delta$,
proving one of the desired estimates. The other one is $(y_{k_i}^-,
y_{k_{i+1}}^-)_{y_{k_i}} \leq 2C_0+3\delta$. To prove it, let us apply
Lemma~\ref{lem:4_points} to the points $y_{k_i}^-, y_{k_i}\fm, y_{k_i}^+,
y_{k_{i+1}}^-$. The Gromov product of the last three is at most
$2C_0+3\delta$ by Lemmas~\ref{lem:dist'}
and~\ref{lem:grom_prod_of_half_space}, and the Gromov product of the first
three is at most $C_0+2\delta$ by applying Lemma~\ref{lem:chaine} to the
reverse of the chain~\eqref{eq:chain_of_removed}. Moreover, the distance
$d(y_{k_i}, y_{k_i}^+)$ is at least $D$, large enough. Therefore,
Lemma~\ref{lem:4_points} indeed applies with $C=2C_0+2\delta$, and gives
$(y_{k_i}^-, y_{k_{i+1}}^-)_{y_{k_i}} \leq 2C_0+3\delta$ as claimed.
\end{proof}

The first point in the previous chain can be replaced with $o$:
\begin{lem}
\label{lem:Pn_chaino'} Let $\bar P_n = \{k_1 < \dotsb < k_p\}$. Then the
sequence $o, y_{k_1}\fm, y_{k_2}^-, y_{k_2}\fm,\dotsc, y_{k_p}\fm, y_{n+1}^-$
is a $(2C_0+4\delta, D - 2C_0 - 3\delta)$-chain. Moreover, $d(o, y_{k_1})
\geq d(o, r_{k_1} \cdot o) + D - C_0-3\delta$.
\end{lem}
\begin{proof}
The only difference compared to the proof of Lemma~\ref{lem:Pn_chaino} is
that we do not have the inequality $(y_{k_1}, o)_{y_{k_1}^-} \leq C_0$ due to
the more complicated definition of refined pivotal times. If we can prove
that $(y_{k_1}, o)_{y_{k_1}^-} \leq C_0 + 3\delta$, the proof of
Lemma~\ref{lem:Pn_chaino} goes through. Let us check this inequality.

As in~\eqref{eq:chain_of_removed}, the points $y_{k_1}^-, y_{k_1}^{(1)},
y_{k_1}^{(3)}, y_{k_1}^{(4)}(=y_{k_1}\fm), y_{k_1}^{(5)}$ form a
$(C_0+\delta, D - 4C_0-6\delta)$ chain. Therefore, $(y_{k_1}^-,
y_{k_1}\fm)_{y_{k_1}^{(1)}} \leq C_0+2\delta$ by Lemma~\ref{lem:chaine}.
Moreover, $(o, y_{k_1}^{(1)})_{y_{k_1}^-} \leq C_0$ by the definition of
pivotal times. As $d(y_{k_1}^{(1)}, y_{k_1}^-) \geq D$ is large, it follows
that Lemma~\ref{lem:4_points} applies to the points $o, y_{k_1}^-,
y_{k_1}^{(1)}, y_{k_1}\fm$ with $C = C_0+2\delta$. It gives $(y_{k_1},
o)_{y_{k_1}^-} \leq C_0+3\delta$, concluding the proof that we have a chain.

Moreover, Lemma~\ref{lem:3_points} together with Lemma~\ref{lem:Pn_chain'}
give
\begin{equation*}
  d(o, y_{k_1}) \geq d(y_{k_1}^-, y_{k_1}\fm) - (o, y_{k_1})_{y_{k_1}^-}
  \geq (d(o, r_{k_1} \cdot o) + D) - (C_0+3\delta),
\end{equation*}
proving the last claim.
\end{proof}

\begin{prop}
\label{prop:doZn'} Let $\bar P_n = \{k_1 < \dotsb < k_p\}$. We have
$d(o,y_{n+1}^-) \geq \sum_i d(o, r_{k_i} \cdot o)$.
\end{prop}
\begin{proof}
This follows from Lemmas~\ref{lem:Pn_chain'} and~\ref{lem:Pn_chaino'}, saying
that we have a chain between $o$ and $y_{n+1}^-$ with jumps of size at least
$d(o, r_{k_i} \cdot o) + D -C_0-3\delta$, and from Lemma~\ref{lem:chaine}
saying that the distance grows at least as the size of the jumps along a
chain.
\end{proof}

To prove Proposition~\ref{prop:complicated_model}, it follows that we should
show that there are many refined pivotal times. For this, we follow the same
strategy as in Subsection~\ref{subsec:pivotal}.

\begin{lem}
\label{lem:prob_new_free'} Fix $s_1,\dotsc, s_n$, and draw $s_{n+1}$
according to $\mu_S^2 * \rho_{n+1} * \mu_S^2$. The probability that
$\card{\bar P_{n+1}} = \card{\bar P_n} + 1$ (i.e., that $n+1$ gets added as a
refined pivotal time) is at least $1- 7\eta$.
\end{lem}
\begin{proof}
In the local geodesic condition, there are 7 alignment conditions to be
satisfied. When drawing $s_{n+1}$ according to $\mu_S^2 * \rho_{n+1} *
\mu_S^2$, each of them is satisfied with probability at least $1-\eta$ (for
each of them, this can be seen by fixing all variables but one and using that
the last one is picked from a Schottky set). Therefore, they are
simultaneously satisfied with probability at least $1-7\eta$.
\end{proof}

To control the backtracking, we defined pivoted sequences. Given $\bar s =
(s_1,\dotsc, s_n)$, let us say that another sequence $\bar s' = (s'_1,
\dotsc, s'_n)$ is pivoted from $\bar s$ if they have the same refined pivotal
times, and $d'_k = d_k$ at all times, and $a'_k = a_k,\ b'_k = b_k,\ r'_k =
r_k,\ c'_k = c_k$ at times which are not a refined pivotal time. In other
words, we freeze the last jump $d_k$, but we keep the freedom in the other
parts of $s_k$ at refined pivotal times only.

The next lemma is proved exactly like Lemma~\ref{lem:many_pivoted}.
\begin{lem}
\label{lem:many_pivoted'} Let $i$ be a refined pivotal time of $\bar s =
(s_1,\dotsc, s_n)$. Replace $s_i = a_i b_i r_i c_i d_i$ with $s'_i = a'_i
b'_i r'_i c'_i d_i$ which still satisfies the local geodesic condition (with
$n$ replaced by $i$). Then $(s_1,\dotsc, s'_i, \dotsc, s_n)$ is pivoted from
$\bar s$.
\end{lem}

Denote by $\bar\boE_n(\bar s)$ the sequences which are pivoted from $\bar s$.
Conditionally on $\bar\boE_n(\bar s)$, the variables $s'_i$ over pivotal
times $i$ are independent, but drawn from distributions that depends on $i$.

\begin{lem}
\label{lem:rembobine'} Let $\bar s = (s_1,\dotsc, s_n)$ be a trajectory with
$q$ refined pivotal times. We condition on $\bar \boE_n(\bar s)$, and we draw
$s_{n+1}$ according to $\mu_S^2 * \rho_{n+1} * \mu_S^2$. Then, for all $j\geq
0$,
\begin{equation*}
  \Pbb(\card{\bar P_{n+1}} < q-j \given \bar\boE_n(\bar s)) \leq (7\eta)^{j+1}.
\end{equation*}
\end{lem}
\begin{proof}
The proof is essentially the same as for Lemma~\ref{lem:rembobine}. Assume
that $s_{n+1}$ is fixed and gives rise to some backtracking. Let us show that
further backtracking happens with probability at most $7\eta$, from which the
estimate follows inductively. Let $m < k$ be the last two refined pivotal
times, and let $x_{i-1}$ be the last point in a chain from $y_m$ to $y_k^-$
witnessing that $y_k^- \in \boC\boS_{y_m}(y_m^+; C_0+\delta)$ as guaranteed
by Lemma~\ref{lem:dist'}.

In $s'_k$, let us condition also with respect to $b'_k, r'_k, c'_k$
compatible with the local geodesic condition. Then the total number of
possible values for $a'_k$ that give rise to $s'_k$ satisfying the local
geodesic condition is at least $(1-\eta)\card{S}$, as one should ensure the
condition $((a'_k)^{-1}\cdot o, b'_k \cdot o)_o \leq C_0$ and $S$ is a
Schottky set. Among these, the values of $a'_k$ that may give rise to further
backtracking are those for which the points $x_{i-1}\fm, y_k^-, y_k^{(1)},
y_{n+1}^-$ are not $C_0$-aligned, because this alignment would imply
$y_{n+1}^- \in \boC\boS_{y_m}(y_m^+; C_0+\delta)$ (as in the proof of
Lemma~\ref{lem:rembobine}) and would block the backtracking. By the Schottky
condition applied twice, there are at most $2\eta\card{S}$ such $a'_k$.
Therefore, the probability of further backtracking is at most $2\eta/(1-\eta)
\leq 7\eta$.
\end{proof}

\begin{lem}
\label{lem:AnU2'} Let $A_n = \card{\bar P_n}$ be the number of pivotal times.
Then, in distribution, $A_{n+1} \geq A_n + U$ where $U$ is a random variable
independent from $A_n$ and distributed as follows:
\begin{align*}
  &\Pbb(U = -j) = (1-7\eta) (7\eta)^j \text{ for $j > 0$},\\
  &\Pbb(U = 0)  = 0,\\
  &\Pbb(U = 1)  = 1-7\eta.
\end{align*}
In other words, $\Pbb(A_{n+1} \geq i) \geq \Pbb(A_n + U \geq i)$ for all $i$.
\end{lem}
\begin{proof}
This is proved exactly like Lemma~\ref{lem:AnU2} using
Lemma~\ref{lem:rembobine'}.
\end{proof}

\begin{prop}
\label{prop:Pnlinear'} There exists $\kappa>0$ only depending on $\eta$ such
that for all $n$,
\begin{equation*}
  \Pbb(\card{\bar P_n} \leq (1-14\eta) n) \leq e^{-\kappa n}.
\end{equation*}
\end{prop}
\begin{proof}
Let $U_1,U_2,\dotsc$ be a sequence of independent copies of the variable $U$
from Lemma~\ref{lem:AnU2'}. Iterating this lemma gives
\begin{equation*}
  \Pbb(\card{\bar P_n} \geq i) \geq \Pbb(U_1+\dotsb+U_n \geq i)
\end{equation*}
for all $i$. In particular, $\Pbb(\card{\bar P_n} \leq (1-14\eta) n) \leq
\Pbb(U_1 + \dotsb + U_n \leq (1-14\eta) n)$. The $U_i$ are real random
variables with an exponential moment, and expectation $(1-14\eta)/(1-7\eta) >
1-14\eta$. Large deviations for sums of i.i.d.~real random variables ensure
that $\Pbb(U_1 + \dotsb + U_n \leq (1-14\eta) n)$ is exponentially small.
\end{proof}

\begin{proof}[Proof of Proposition~\ref{prop:complicated_model}]
We want to bound $\Pbb(d(o, y_{n+1}^-) \leq M)$. By
Proposition~\ref{prop:Pnlinear'}, we have
\begin{equation}
\label{eq:do_bound_add_exp}
  \Pbb(d(o, y_{n+1}^-) \leq M) \leq \Pbb(d(o, y_{n+1}^-) \leq M, \card{\bar P_n} \geq (1-14\eta) n) + e^{-\kappa n}.
\end{equation}
Therefore, we may focus on trajectories with $\card{\bar P_n} \geq (1-14\eta)
n$. Let $\bar s = (s_1,\dotsc, s_n)$ be such a trajectory, and $\bar
\boE_n(\bar s)$ its equivalence class under the pivotal relation. We will
estimate $\Pbb(d(o, y_{n+1}^-) \leq M \given \bar \boE_n(\bar s))$.

Along $\bar\boE_n(\bar s)$, we have $d(o, y_{n+1}^-) \geq \sum_{i=1}^p d(o,
r_{k_i} \cdot o)$ where the pivotal times are $k_1 < \dotsc < k_p$, by
Proposition~\ref{prop:doZn'}. As $p \geq (1-14\eta) n$, we obtain in
particular
\begin{equation}
\label{eq:do_dominates_sum_r}
  d(o, y_{n+1}^-) \geq \sum_{i=1}^{\lfloor (1-14\eta)n\rfloor} d(o, r_{k_i} \cdot o).
\end{equation}
Along $\bar\boE_n(\bar s)$, the random variables $r_{k_i}$ are independent
(as what happens at different pivotal times is independent by construction),
but they are not distributed like $\rho_{k_i}$ a priori, since the local
geodesic condition may twist its distribution. Denoting by $L_{k_i}$ the set
of $(a, b, r, c)$ that satisfy the local geodesic condition, then the
distribution of $(a, b, r, c)$ is $(\mu_S^2 * \rho_{k_i} * \mu_S) 1_{L_{k_i}}
/ (\mu_S^2 * \rho_{k_i} * \mu_S)(L_{k_i})$. In particular, the probability
that $r_{k_i}$ equals a given $r$ is
\begin{multline*}
  \rho_{k_i}(r) \mu_S^3\{(a, b, c) \text{ such that } (a, b, r, c) \in L_{k_i}\} / (\mu_S^2 * \rho_{k_i} * \mu_S)(L_{k_i})
  \\
  \geq \rho_{k_i}(r) \mu_S^3\{(a, b, c) \text{ such that } (a, b, r, c) \in L_{k_i}\}.
\end{multline*}
Once $r$ is fixed, there are $6$ alignment relations to be satisfied for $a,
b, c$ to make sure that $(a, b, r, c)$ satisfies the local geodesic
condition. Each of them is satisfied with probability at least $1-\eta$, so
we get $\mu_S^3\{(a, b, c) \text{ such that } (a, b, r, c) \in L_{k_i}\} \geq
1-6\eta$. Finally,
\begin{equation*}
  \Pbb(r_{k_i} = r \given \bar\boE_n(\bar s)) \geq (1-6\eta) \rho_{k_i}(r).
\end{equation*}
As the distance $d(o, r\cdot o)$ for $r$ drawn according to $\rho_{k_i}$
dominates the random variable $R$ in the assumptions of the lemma, it follows
that the conditional distribution in $\bar\boE_n(\bar s)$ dominates $BR$,
where $B$ is a Bernoulli random variable, equal to $1$ with probability
$1-6\eta$ and to $0$ with probability $6\eta$. Conditionally on
$\bar\boE_n(\bar s)$, it follows from~\eqref{eq:do_dominates_sum_r} that
$d(o, y_{n+1}^-)$ dominates $\sum_{i=1}^{\lfloor (1-14\eta)n\rfloor} B_i
R_i$. As this estimate is uniform over the equivalence classes, we get
from~\eqref{eq:do_bound_add_exp} the inequality
\begin{equation*}
  \Pbb(d(o, y_{n+1}^-) \leq M) \leq
  \Pbb\pare*{\sum_{i=1}^{\lfloor (1-14\eta)n\rfloor} B_i R_i\leq M} + e^{-\kappa n}.
\end{equation*}

Since the $B_i$ have expectation $1-6\eta$, the probability
$\Pbb(\sum_{i=1}^n B_i \leq (1-7\eta)n)$ is exponentially small. We get
\begin{equation*}
  \Pbb(d(o, y_{n+1}^-) \leq M) \leq \Pbb\pare*{\sum_{i=1}^{\lfloor (1-14\eta)n\rfloor} B_i R_i\leq M, \sum_{i=1}^n B_i \geq (1-7\eta) n} + e^{-\kappa' n}.
\end{equation*}
To estimate the probability on the right, let us condition with respect to
the $B_i$. There are at most $7\eta n$ of them that vanish. Therefore, $\sum
B_i R_i$ is a sum of at least $(1-21\eta)n$ independent copies of $R$, and
the probability that the sum is at most $M$ is bounded by
$\Pbb(\sum_{i=1}^{\lfloor (1-21\eta)n\rfloor} R_i\leq M)$. As this estimate
is uniform over the choice of the $B_i$s, this concludes the proof.
\end{proof}

\subsection{Precise estimates for walks without first moment}

In this paragraph, we consider a discrete probability measure $\mu$ on the
set of isometries of $X$ which has no first moment: $\E(d(o, g\cdot o)) =
\infty$ when $g$ is drawn according to $\mu$. We will prove
Theorems~\ref{thm:A} and~\ref{thm:B} under this assumption. It suffices to
prove the latter, as the former follows readily.

Let $r>0$ be arbitrary. We have to show the existence of $\kappa > 0$ such
that
\begin{equation*}
  \Pbb((Z_n \cdot o, Z_\infty)_o \leq r n) \leq e^{-\kappa n}.
\end{equation*}
Let $\eta = 1/100$. Let $S$ be an $(\eta, C_0, D)$-Schottky set in the
support of $\mu^M$ for some $M>0$, where $D$ is large enough compared to
$C_0$, as given by Corollary~\ref{cor:exists_Schottky}. We follow the
construction in Paragraph~\ref{subsec:thm0} to reconstruct the $\mu$-random
walk, except that instead of sampling the specific jumps from $\mu_S^2$, we
will sample them from $\mu_S^2 * \mu * \mu_S^2$: for $N=4M+1$ and some
$\alpha>0$, we may write $\mu^N = \alpha \mu_S^2 * \mu * \mu_S^2  +
(1-\alpha) \nu$ for some probability measure $\nu$, where $\mu_S$ is the
uniform measure on $S$.

The random walk is reconstructed by starting from Bernoulli random variables
$\epsilon_i$ (satisfying $\Pbb(\epsilon_i=1) = \alpha$ and $\Pbb(\epsilon_i =
0) = 1-\alpha$), and sampling from $\mu_S^2 * \mu * \mu_S^2$ when
$\epsilon_i=1$ and from $\nu$ when $\epsilon_i = 0$. Conditioning on
$(\epsilon_i)$ and on the jumps when $\epsilon_i = 0$, we are left with a
walk as in Proposition~\ref{prop:complicated_model}. For this walk, we define
a sequence of refined pivotal times as in
Subsection~\ref{subsec:complicated}. Let $\tau=\tau(n)$ be the last index $j$
such that $N(t_j+1) \leq n$, so that the interval $[Nt_j, N(t_j+1))$ is
contained in $[0, n)$. Then the sequence of refined pivotal times associated
to the walk until time $n$ has the form $\bar P_1, \bar P_2,\dotsc, \bar
P_{\tau-1}, \bar P_\tau^{(n)}$. Moreover, $u_n \coloneqq \card{\bar
P_{\tau(n)}^{(n)}}$ satisfies
\begin{equation}
\label{eq:Pun_refined}
  \Pbb(u_n \leq \kappa n) \leq e^{-\kappa n},
\end{equation}
for some $\kappa > 0$: this is proved as
Proposition~\ref{prop:un_exponential}, just using
Proposition~\ref{prop:Pnlinear'} instead of Proposition~\ref{prop:Pnlinear}
inside the proof.

Assume now that the walk converges at infinity (this is true almost
everywhere) and that $u_k > \kappa k$ for all $k \geq n$ (this is true
outside of a set of exponentially small measure, by summing the estimates
in~\eqref{eq:Pun_refined}). Let $x=x_n$ be the position of the walk at the
$(\lfloor\kappa n\rfloor-1)$-th refined pivotal time in $ \bar
P_{\tau(n)}^{(n)}$. Then for all $k\geq n$, the point $Z_k \cdot o$ belongs
to the $(2C_0+6\delta)$-shadow of $x$ seen from $o$ (this is proved just like
Lemma~\ref{lem:stuck_in_shadow}, using Lemma~\ref{lem:Pn_chaino'}). As
in~\eqref{eq:linear_lower_bound}, this implies the inequality
\begin{equation*}
  (Z_n \cdot o, Z_\infty)_o \geq d(o, x_n) - (2C_0+9\delta).
\end{equation*}
Finally, we have
\begin{equation*}
  \Pbb((Z_n \cdot o, Z_\infty)_o \leq r n) \leq e^{-\kappa n} + \Pbb(u_n\geq \kappa n, \ d(o, x_n) \leq r n + (2C_0+9\delta)).
\end{equation*}

Let us estimate the rightmost probability. We condition on the $(\epsilon_i)$
(which fixes $\tau$) and on the jumps when $\epsilon_i = 0$, to be in the
setting of Subsection~\ref{subsec:complicated}. As $x$ is one of the points
$y_{k+1}^-$ for $(\kappa/2) n \leq k \leq n$, we can sum the estimates of
Proposition~\ref{prop:complicated_model} (applied to $k$ instead of $n$), to
get a bound of the form
\begin{equation*}
  n \Pbb(R_1+\dotsb + R_{\lfloor (1-21\eta) (\kappa/2) n\rfloor} \leq (r+1)n),
\end{equation*}
where the $R_i$ are independent random variables distributed like $d(o,
g\cdot o)$ where $g$ is drawn according to $\mu$. Letting $\beta =
(1-21\eta)(\kappa/2) > 0$, we get
\begin{equation*}
  \Pbb((Z_n \cdot o, Z_\infty)_o \leq r n) \leq e^{-\kappa n} + n \Pbb(R_1+\dotsb + R_{\lfloor \beta n\rfloor} \leq (r+1)n).
\end{equation*}

Since we are assuming that $\mu$ has no first moment, the nonnegative random
variables $R_i$ are not integrable. Applying the usual large deviations
estimate to a truncated version of $R$, we deduce that for any $A>0$ there
exists $c(A)$ such that $\Pbb(R_1 + \dotsb + R_k \leq A k) \leq e^{-c(A) k}$.
Together with the previous equation, this gives an exponential bound on
$\Pbb((Z_n \cdot o, Z_\infty)_o \leq r n)$. This concludes the proof of
Theorem~\ref{thm:B} (and therefore also of Theorem~\ref{thm:A}) when there is
no first moment. \qed

\subsection{Precise estimates for walks with a first moment}

\label{subsec:precise_moment}

Assume now that $\mu$ is a measure with a first moment. Then $\E_{\mu^n}(d(o,
g\cdot o))/n$ converges by subadditivity to a limit $\ell$, the escape rate
of the walk. Let $r<\ell$. Our goal in this paragraph is to prove
Theorem~\ref{thm:B} (and therefore also Theorem~\ref{thm:A}) in this setting:
we will show that, for some $\kappa>0$, we have
\begin{equation*}
  \Pbb((Z_n \cdot o, Z_\infty)_o \leq r n) \leq e^{-\kappa n}.
\end{equation*}

To prove this estimate, we will again use the refined model of
Subsection~\ref{subsec:complicated}, but we will have to do so in a careful
enough way.

Fix $\eta>0$ small enough depending only on $r$ and $\ell$ (how small will be
prescribed at the very end of the proof). By
Corollary~\ref{cor:exists_Schottky}, there exists an $(\eta, C_0,
D)$-Schottky set $S$ in the support of $\mu^M$ for some $M>0$, where $D$ is
large enough compared to $C_0$. For $N = 2M$, we may write $\mu^N = \alpha
\mu_S^2 + (1-\alpha)\nu$ for some probability measure $\nu$. Replacing
$\alpha$ with $\alpha/2$ if necessary, we can also assume that $\nu$ is
non-elementary.

Let us now fix $A>0$ very large (how large will be described in the course of
the proof, depending on $\eta$, $\alpha$ and $\nu$). Let $\epsilon_i$ be a
sequence of Bernoulli random variables, equal to $1$ with probability
$\alpha$ and to $0$ with probability $1-\alpha$. Define inductively a
sequence of times $t_1, t'_1, t_2, t'_2, \dotsc$ as follows. First, $t_1$ is
the first time with $\epsilon_{t_1} = 1$. Then $t'_1$ is the smallest time
$>t_1 + A$ with $\epsilon_{t'_1} = 1$. Then $t_2$ is the smallest time
$>t'_1$ with $\epsilon_{t_2}=1$. And so on, picking the first times where
$\epsilon_i = 1$ but keeping a gap at least $A$ between $t_i$ and $t'_i$.
Then, pick $\gamma_n$ distributed according to the following measure: if $n$
is of the form $t_i$ or $t'_i$, use $\mu_S^2$. If $n$ is in $[t_i+1, t_i+A]$,
use $\mu^{N}$. Otherwise, use $\nu$.
\begin{claim}
With this construction, $\gamma_0 \dotsm \gamma_{n-1}$ is distributed like
$Z_{Nn}$.
\end{claim}
\begin{proof}
Conditionally on the $\epsilon_0,\dotsc, \epsilon_{n-1}$ and on $\gamma_0,
\dotsc, \gamma_{n-1}$, we will show that $\gamma_n$ is distributed according
to $\mu^N$, from which the result follows. Consider the maximal $t_j$ or
$t'_j$ before $n$. If it is a $t_j$ and $n \leq t_j+A$, then $\gamma_n$ is
picked according to $\mu^N$ by definition, and there is nothing left to
prove. Otherwise, the choice of the measure for $\gamma_n$ depends on
$\epsilon_n$: we use $\mu_S^2$ if $\epsilon_n=1$ (with probability $\alpha$)
or $\nu$ if $\epsilon_n=0$ (with probability $1-\alpha$). Altogether,
$\gamma_n$ is drawn according to $\alpha \mu_S^2 + (1-\alpha)\nu=\mu^N$,
proving the claim.
\end{proof}
With a standard coupling argument, extending $\Omega$ if necessary, we can
also construct on $\Omega$ a sequence of independent random variables
$g_0,g_1,\dotsc$ with distribution $\mu$ such that $\gamma_i = g_{iN}\dotsm
g_{iN + N-1}$.

\medskip

The intuition behind the use of this decomposition is the following. Since
$\alpha$ is possibly small, the times with $\epsilon_i = 1$, which have
frequency $1/\alpha$, may be sparse. However, if $A$ is much larger than
$1/\alpha$, the waiting time between $t_i+A$ and $t'_i$, or between $t'_i$
and $t_{i+1}$, will be comparatively much shorter. Therefore, the walk will
be essentially a concatenation of jumps corresponding to $\mu^{NA}$. These
jumps essentially go in independent directions (this is formalized precisely
by Proposition~\ref{prop:complicated_model}), so the size of the walk at time
$NAk$ will be bounded below by the sum of $(1-21\eta)k$ independent random
variables distributed like jumps of $\mu^{NA}$, which are of order $NA \ell$.
Altogether, the probability to have size smaller than $(1-21\eta)NAk \ell$ at
time roughly $NAk$ will be exponentially small, proving Theorem~\ref{thm:A}
in this setting.

To make this precise, we will need to control quantitatively the waiting
times. Also, the distribution of the jumps between $t_i$ and $t'_i$ is not
$\mu^{NA}$, but $\mu^{NA} * \nu^{t'_i-(t_i+A)}$. We will have to show that
the jumps of this family of measures are uniformly controlled from below, to
be able to apply Proposition~\ref{prop:complicated_model}. Note that this
application motivates why we had to formulate this proposition using
different measures $\rho_i$ for the different jumps, instead of one single
measure $\rho$.

\medskip

Let us start the proof, adapting the formalism of
Subsection~\ref{subsec:thm0} to our current setting. Fix $n \in \N$. We let
$\tau=\tau(n)$ be the last index $j$ such that $N(t'_j+1) \leq n$, so that
the interval $[Nt_j, N(t'_j+1))$ is contained in $[0, n)$. We will decompose
the product $g_0\dotsm g_{n-1}$ as a product of the elements $s'_j$ (the
product of all $g_i$ for $i \in [Nt_j, N(t'_j+1))$) interspersed with other
words that we will consider as fixed, to be in the framework of
Subsection~\ref{subsec:complicated}. Let $w_j = g_{N(t'_{j}+1)} \dotsm
g_{Nt_{j+1} -1}$ (where by convention $t'_0 = 0$), and let $w' = w'(n) =
g_{N(t'_{\tau(n)}+1)} \dotsm g_{n-1}$ be the last missing word (it really
depends on $n$, contrary to the previous words that just fill the gaps
between blocks $[t_j, t'_j]$). By construction,
\begin{equation*}
  Z_n \cdot o = w_0 s'_1 w_1 \dotsm w_{\tau -1} s'_\tau w'(n) \cdot o.
\end{equation*}

We can associate to this decomposition a sequence of refined pivotal times
$\bar P_1^{(n)}, \dotsc, \bar P_\tau^{(n)}$, where the exponent ${}^{(n)}$ is
here to emphasize that the intermediate words we use depend on $n$. In fact,
the only word that really depends on $n$ is the last word $w' = w'(n)$, as
the other ones are $w_j = g_{(N+1)t_{j}} \dotsm g_{Nt_{j+1} -1}$ so they only
depend on $t_j$. Hence, the sequence of refined pivotal times is rather
\begin{equation*}
  \bar P_1, \bar P_2,\dotsc, \bar P_{\tau-1}, \bar P_\tau^{(n)}.
\end{equation*}

If we condition on the $\epsilon_i$ (which fixes the $t_i$ and $t'_i$), and
on the $g_i$ for $i$ not belonging to $\bigcup [N t_j, N (t'_j+1))$ (which
fixes the $w_i$ and $w'(n)$), then we are in the setting of
Proposition~\ref{prop:complicated_model}, with $\rho_i = \mu^{NA} * \nu^{t'_j
- (t_j+A)}$. To apply this proposition, we need to check that jumps with
respect to such a measure are uniformly bounded below.

\begin{lem}
\label{lem:big_jumps_nui} Assume that $A$ is large enough. Let $R_{NA}$ be
the distribution of the size of jumps for $\mu^{NA}$. Let $B$ be a Bernoulli
random variable, equal to $1$ with probability $1-\eta$ and to $0$ with
probability $\eta$, independent of $R_{NA}$. Then, for any $i \geq 0$, for
any $M\geq 0$,
\begin{equation*}
  \Pbb_{\mu^{NA} * \nu^i} (d(o, g\cdot o) \geq M) \geq \Pbb(BR_{NA} \geq M + \eta NA).
\end{equation*}
In other words, the jumps for $\mu^{NA} * \nu^i$ dominate stochastically
$BR_{NA} - \eta NA$, uniformly in $i$.
\end{lem}
\begin{proof}
We have
\begin{equation*}
  \Pbb_{\mu^{NA} * \nu^i} (d(o, g\cdot o) \geq M) = \sum_h \mu^{NA}(h) \Pbb_{\nu^i} (d(o, hg \cdot o) \geq M).
\end{equation*}
By Lemma~\ref{lem:uniform_control_away} applied to the nonelementary measure
$\nu$ and to $\epsilon=\eta$, there exists $C>0$ such that, uniformly in $h$,
with probability at least $1-\eta$ with respect to $\nu^i$ for $g$ one has
$d(o, hg \cdot o) \geq d(o, h\cdot o) - C$. This gives
\begin{equation*}
  \Pbb_{\nu^i} (d(o, hg \cdot o) \geq M) \geq (1-\epsilon)1_{d(o, h\cdot o) \geq M + C}.
\end{equation*}
Therefore,
\begin{align*}
  \Pbb_{\mu^{NA} * \nu^i} (d(o, g\cdot o) \geq M)
  &\geq \sum_{d(o, h\cdot o) \geq M + C} \mu^{NA}(h)(1-\epsilon)
  = (1-\epsilon) \Pbb_{\mu^{NA}}(d(o, h\cdot o) \geq M + C)
  \\&
  = (1-\epsilon) \Pbb(R_{NA} \geq M + C)
  =\Pbb(B R_{NA} \geq M + C).
\end{align*}
Taking $A$ large enough so that $\eta NA \geq C$, this is bounded from below
by $\Pbb(B R_{NA} \geq M + \eta NA)$.
\end{proof}

From now on, we will assume that $A$ is large enough so that
Lemma~\ref{lem:big_jumps_nui} holds.

\begin{lem}
\label{lem:bound_tau_refined} Assume that $A$ is large enough. The sequence
$\tau(n)$ grows like $n/(NA)$ with high probability. More precisely, there
exists $c>0$ such that
\begin{equation*}
  \Pbb(\tau(n) \leq (1-\eta) n/(NA)) \leq e^{-c n}.
\end{equation*}
\end{lem}
\begin{proof}
We have
\begin{equation*}
  t'_j = A j + \sum_{i=1}^j (t'_i - (t_i + A)) + \sum_{i=1}^j (t_i - t'_{i-1}).
\end{equation*}
The random variables $t'_i - (t_i + A)$ and $t_i-t'_{i-1}$ are independent
and have an exponential tail (just depending on $\alpha$). Therefore, there
exists $C>0$ and $c>0$ (not depending on $A$) such that
\begin{equation*}
  \Pbb\pare*{\sum_{i=1}^j (t'_i - (t_i + A)) + \sum_{i=1}^j (t_i - t'_{i-1}) \geq Cj} \leq e^{-cj}.
\end{equation*}
Outside of a set $O_j$ with exponentially small probability, we obtain $t'_j
\leq Aj + Cj$. Therefore, $N(t'_j + 1) \leq N (Aj + Cj+1)$, which is bounded
by $N A j/(1-\eta)$ if $A$ is large enough compared to $C$. Take $j = j(n) =
\lfloor (1-\eta)n/(NA)\rfloor$. It satisfies $N A j/(1-\eta) \leq n$. On the
complement of $O_j$, we have $N(t'_j + 1) \leq n$, and therefore $\tau(n)
\geq j$. Hence, the inequality $\tau(n) \leq (1-\eta) n/(NA)$ can only hold
on $O_j$, whose probability is exponentially small in terms of $n$.
\end{proof}

Let $u_n \coloneqq \card{\bar P_\tau^{(n)}}$ be the number of refined pivotal
times up to time $n$.

\begin{lem}
\label{le:Pbbun} There exists $c>0$ such that $\Pbb(u_n \leq (1-15\eta)
n/(NA)) \leq e^{-c n}$.
\end{lem}
\begin{proof}
By Lemma~\ref{lem:bound_tau_refined}, we have
\begin{equation*}
  \Pbb(u_n \leq (1-15\eta) n/(NA)) \leq e^{-c n} + \Pbb(u_n \leq (1-15\eta) n/(NA), \tau(n) \geq (1-\eta) n/(NA)).
\end{equation*}
Let us concentrate on the second set. We condition with respect to
$\epsilon_i$ (which fixes the $t_i$, the $t'_i$, and $\tau$) and with respect
to the $g_i$ outside of the intervals $[Nt_j, N(t'_j+1))$ (which fixes the
$w_j$ and $w'$). Once these are fixed, we are in the framework of
Subsection~\ref{subsec:complicated}. We may therefore apply
Proposition~\ref{prop:Pnlinear'} and deduce that, conditionally on these
quantities, we have $\Pbb(u_n \leq (1-14\eta) \tau) \leq e^{-c \tau}$, for
some $c>0$. As $\tau \geq (1-\eta) n/(NA)$, this gives conditionally
$\Pbb(u_n \leq (1-\eta) (1-14\eta) n/(NA)) \leq e^{-c(1-\eta) n/(NA)}$. As
$1-15\eta \leq (1-\eta) (1-14\eta)$ and the previous bound is uniform on the
conditioning, this implies the conclusion.
\end{proof}

Assume now that $Z_k\cdot o$ converges to a point $Z_\infty$ at infinity and
moreover, for all $k\geq n$, holds $u_k \geq (1-15\eta) k/(NA)$ (this happens
outside of a set of exponentially small probability, by
Lemma~\ref{le:Pbbun}). Let $\bar t = \bar t(n) = \lfloor
(1-16\eta)n/(NA)\rfloor < \card{P_\tau^{(n)}}$, and let $x=x_n$ be the
position of the walk at the $\bar t$-th refined pivotal time. An adaptation
of Lemma~\ref{lem:stuck_in_shadow} to this setting (based on
Lemma~\ref{lem:Pn_chaino'}) shows that, for all $k\geq n$, the point $Z_k
\cdot o$ belongs to the $(2C_0+6\delta)$-shadow of $x$ seen from $o$. In
turn, as in~\eqref{eq:linear_lower_bound}, this implies the inequality
\begin{equation*}
  (Z_n \cdot o, Z_\infty)_o \geq d(o, x_n) - (2C_0+9\delta).
\end{equation*}
Finally, we have
\begin{equation*}
  \Pbb((Z_n \cdot o, Z_\infty)_o \leq r n) \leq e^{-c n} + \Pbb(d(o, x_n) \leq r n + (2C_0+9\delta)).
\end{equation*}
For large enough $n$, we have $r n + (2C_0+9\delta) \leq (r+\eta)n$. Together
with Lemma~\ref{lem:bound_tau_refined}, we get
\begin{equation*}
  \Pbb((Z_n \cdot o, Z_\infty)_o \leq r n) \leq e^{-c n} + \Pbb(d(o, x_n) \leq (r+\eta) n, \tau(n) \geq (1-15\eta) n/(NA)).
\end{equation*}
for some $c > 0$.

To conclude, it suffices to show that the right-most probability is
exponentially small. Let us condition on the $\epsilon_i$ (which fixes the
$t_i$, the $t'_i$ and $\tau$) and on the $g_i$ for $i$ not belonging to
$\bigcup [N t_j, N (t'_j+1))$, to be again in the setting of
Subsection~\ref{subsec:complicated}. Note that $\bar t$ is not fixed by this
conditioning. However, $x_n$ is one of the points $y_{m+1}^- = w_0 s'_1
\dotsm s'_k w_m$, for some $m \geq (1-16\eta)n/(NA)$. We claim that it
suffices to show that, for such an $m$, we have
\begin{equation}
\label{eq:qsdfqsdfsd}
  \Pbb(d(o, y_{m+1}^-) \leq (r+\eta)n) \leq e^{-cm}.
\end{equation}
Indeed, the right hand side is exponentially small in terms of $n$. Summing
over $m \in [(1-16\eta)n/(NA), n/(NA)]$, we get a bound at most $n e^{-c'n}$,
which is again exponentially small as desired.

To prove the inequality~\eqref{eq:qsdfqsdfsd}, we apply
Proposition~\ref{prop:complicated_model}, at the time $m$.
Lemma~\ref{lem:big_jumps_nui} shows that the stochastic domination
assumptions of this lemma are satisfied, for $R = B R_{NA} - NA\eta$ where
$B$ is a $(1-\eta)$-Bernoulli random variable. This proposition gives
\begin{equation*}
  \Pbb(d(o, y_{m+1}^-) \leq (r+\eta)n) \leq \Pbb(R_1 + \dotsb + R_{\lfloor (1-21\eta)m\rfloor} \leq (r+\eta) n) + e^{-c m},
\end{equation*}
where the $R_i$ are independent copies of $R$. The last term is compatible
with~\eqref{eq:qsdfqsdfsd}. For the first term, we will apply large
deviations for sums of i.i.d.\ real random variables. We have
\begin{equation*}
  \E(R_i) = \E(R) = (1-\eta) \E(R_{NA}) - NA\eta \geq (1-\eta) NA\ell -\eta NA,
\end{equation*}
as $\E(R_{NA})/(NA)$ is the average drift at time $NA$, which converges to
$\ell$ from above by subadditivity. For $z = (1-\eta) NA\ell -2\eta NA <
\E(R)$, large deviations ensure that $\Pbb(R_1 + \dotsb + R_k \leq z k)$ is
exponentially small in terms of $k$. Therefore, it is enough to show that $(r
+ \eta) n \leq z (1-21\eta)m$ to conclude. As $m \geq (1-16\eta)n/(NA)$, we
have
\begin{align*}
  \frac{(r+\eta) n}{z (1-21\eta) m} & \leq \frac{(r+\eta) n}{((1-\eta) NA\ell -2\eta NA) (1-21\eta) (1-16\eta)n/(NA)}
  \\& = \frac{r+\eta}{((1-\eta) \ell -2\eta) (1-21\eta) (1-16\eta)}.
\end{align*}
When $\eta$ converges to $0$, this converges to $r/\ell<1$. Therefore, for
small enough $\eta$, it is $\leq 1$ as desired. This concludes the proof of
Theorem~\ref{thm:B} when $\mu$ has a first moment. \qed

\subsection{Continuity of the escape rate}

As an illustration of the power of the tools we have introduced above, we can
recover the fact that the rate of escape $\ell(\mu)$ depends continuously on
the measure $\mu$, a fact that was originally proved in hyperbolic groups by
Erschler and Kaimanovich in~\cite{erschler_kaim} (and which, in the general
setting of non-proper hyperbolic spaces, follows from their proof together
with the tools of~\cite{maher_tiozzo}).

\begin{prop}
Consider a discrete non-elementary measure $\mu$ on the space of isometries
of a Gromov-hyperbolic space $X$ with a basepoint $o$. Let $r<\ell(\mu)$.
There exist $\epsilon>0$ and a finite subset $K$ of the support of $\mu$ with
the following property. Let $\mu'$ be a probability measure with $\mu'(g)
\geq \mu(g) - \epsilon$ for all $g\in K$. Then $\ell(\mu') \geq r$.

Even more, there exists $\kappa>0$ such that, for any $\mu'$ as above, the
corresponding random walk $Z'_n$ satisfies for any $n\in \N$ the inequality
\begin{equation}
\label{eq:dZ'n}
  \Pbb(d(o, Z'_n\cdot o) \leq rn) \leq e^{-\kappa n}.
\end{equation}
\end{prop}
Indeed, all the constants in the proofs in
Subsection~\ref{subsec:precise_moment} are completely explicit. Once $K$ is
chosen large enough and $\epsilon$ small enough to ensure that $\mu'$ gives a
weight bounded from below to all the elements in the Schottky set $S$ chosen
at the beginning of this subsection, then all the estimates go through for
$\mu'$ just like for $\mu$. In the end, this gives~\eqref{eq:dZ'n} with a
uniform $\kappa$. This exponential estimate implies $\ell(\mu') \geq r$ as
$d(o, Z'_n \cdot o)/n$ converges almost surely to $\ell(\mu')$.

It follows from the proposition that, when $\mu_n$ converges simply to $\mu$,
then $\liminf \ell(\mu_n) \geq \ell(\mu)$. This is the nontrivial direction
to prove that $\ell(\mu_n) \to \ell(\mu)$, as the other one follows from
subadditivity (as $\ell(\mu') = \Inf_n (\E(d(o, Z'_n \cdot o))/n)$, and each
of these quantities when $n$ is fixed is continuous in $\mu'$ for the $L^1$
topology). We obtain the following corollary.

\begin{cor}
Consider a discrete non-elementary measure $\mu$ on the space of isometries
of a Gromov-hyperbolic space $X$ with a basepoint $o$, and a sequence of
probability measures $\mu_n$ converging to $\mu$ in the $L^1$-sense, i.e.,
$\sum_g d(o, g\cdot o) \abs{\mu_n(g) - \mu(g)} \to 0$. Then $\ell(\mu_n)$
tends to $\ell(\mu)$.
\end{cor}

\bibliography{biblio}
\bibliographystyle{amsalpha}
\end{document}